\documentclass[a4paper,11pt]{amsart}

\usepackage{amsmath, tabularx}
\usepackage{amsthm}
\usepackage{amssymb}
\usepackage{amsfonts}
\usepackage{paralist}
\usepackage{aliascnt}
\usepackage{amsrefs}
\usepackage{amscd}
\usepackage{blkarray}
\usepackage{float}
\usepackage[inner=3.7cm,outer=3.7cm, bottom=3.5cm]{geometry}

\usepackage{setspace}
\usepackage[colorlinks=true,linkcolor=blue,urlcolor=blue]{hyperref}
\usepackage{tikz}
\usetikzlibrary{matrix}
\usetikzlibrary{arrows}

\def\equationautorefname~#1\null{(#1)\null}

\newtheorem{theorem}{Theorem}[section]

\newtheorem{headtheorem}{Theorem}
\newaliascnt{cor}{theorem}
\newtheorem{cor}[cor]{Corollary}
\aliascntresetthe{cor}

\newaliascnt{lem}{theorem}
\newtheorem{lem}[lem]{Lemma}
\aliascntresetthe{lem}

\newaliascnt{conjecture}{theorem}

\aliascntresetthe{conjecture}

\newaliascnt{prop}{theorem}
\newtheorem{prop}[prop]{Proposition}
\aliascntresetthe{prop}

\newaliascnt{dfn}{theorem}
\newtheorem{dfn}[dfn]{Definition}
\aliascntresetthe{dfn}

\newaliascnt{nota}{theorem}
\newtheorem{nota}[nota]{Notation}
\aliascntresetthe{nota}

\newaliascnt{exmp}{theorem}
\newtheorem{exmp}[exmp]{Example}
\aliascntresetthe{exmp}

\newaliascnt{rem}{theorem}
\newtheorem{rem}[rem]{Remark}
\aliascntresetthe{rem}

\newaliascnt{problem}{theorem}

\aliascntresetthe{problem}

\newaliascnt{construction}{theorem}

\aliascntresetthe{construction}

\newaliascnt{defprop}{theorem}

\aliascntresetthe{defprop}

\newaliascnt{obs}{theorem}
\newtheorem{obs}[obs]{Observation}
\aliascntresetthe{obs}

\DeclareMathOperator{\regularity}{reg}

\newcommand{\edge}[1]{I(#1)}
\newcommand{\cncm}{C_n\cdot P_2\cdot C_m}
\newcommand{\cnplcm}{C_n\cdot P_l\cdot C_m }
\newcommand{\cnlcm}[1]{C_n\cdot P_#1\cdot C_m }
\newcommand{\cn}{C_n}
\newcommand{\cm}{C_m}
\newcommand{\Mod}{\normalfont\text{mod}\;}
\newcommand{\reg}[1]{\regularity{#1}}
\newcommand{\matching}{\mathcal{M}}

\newcommand{\LozGraph}{\mathcal{L}_x(G)}

\begin{document}

\title{Regularity of bicyclic graphs and their powers}

\author{Yairon Cid-Ruiz}
\address{Department de Matem\`{a}tiques i Inform\`{a}tica, Facultat de Matem\`{a}tiques i Inform\`{a}tica, Universitat de Barcelona, Gran Via de les Corts Catalanes, 585; 08007 Barcelona, Spain.}
\email{ycid@ub.edu}
\thanks{The first named author was funded by the European Union's Horizon 2020 research and innovation programme under the Marie Sk\l{}odowska-Curie grant agreement No. 675789.
}

\author{Sepehr Jafari}
\address {Dipartimento di Matematica, Università degli studi di Genova, Via Dodecaneso, 35, 16146 Genova GE, Italy}
\email{sepehr@dima.unige.it}

\author{Navid Nemati}
\address{Institut de Math\'{e}matiques de Jussieu, UPMC, 75005 Paris, France}
\email{navid.nemati@imj-prg.fr}

\author{Beatrice Picone}
\address{Dipartimento di Matematica e Informatica\\
		Viale A. Doria, 6 - 95100 - Catania, Italy}
\email{picone@dmi.unict.it}



\begin{abstract}
Let $I(G)$ be the edge ideal of a bicyclic graph. 
In this paper, we characterize the Castelnuovo-Mumford regularity of $I(G)$ in terms of the induced matching number  of $G$.
For the base case of this family of graphs, i.e. dumbbell graphs, we explicitly compute the induced matching number. Moreover, we prove that  $\reg{I(G)^q}=2q+\reg{I(G)}-2$, for all $ q\geq 1 $, when $ G $ is a dumbbell graph with a connecting path having no more than two vertices. 
\end{abstract}

\maketitle
\vspace*{-1.001cm}
\section*{Introduction}

Let $ I $ be a homogeneous  ideal of the polynomial ring $ R=K[x_1,\ldots,x_r] $. The Castelnuovo-Mumford regularity of $I$, denoted by $ \reg{(I)} $, has been an interesting
and active research topic for the past decades.
There exists a vast literature on the study of the $ \reg{(I)} $. 
One of the most important results on the behavior of the regularity of powers of ideals was given  independently by Cutkosky, Herzog, and Trung in \cite{Assomptatic_reg_Herzog}, and  by Kodiyalam in  \cite{Kodiyalam}. 
In both papers, it is proved that for all $ q\geq q_0 $, the regularity of powers of $ I $ is asymptotically a linear function $\reg{(I^q)}=dq+b$, where $q_0$ is the so-called stabilizing index, and $b$ is the so-called constant. 
The value of $d$ in the above formula is well understood. 
For example, $d$ is equal to the degree of the generators of $I$ when $I$ is equigenerated. 
However, their method does not give precise information on $ q_0 $  and $ b $.
 
Since then, many researchers have tried to compute $ q_0 $ and $ b $ for special families of ideals. 
The most simple case, yet interesting, is when $ I $ is the edge ideal of a finite simple graph. 
Let $G=(V(G), E(G))$ denote a finite simple undirected graph. 
Let $R$ be the polynomial ring $K[x_i \mid x_i \in V(G)]$ where  $K$ is any field. 
The edge ideal $I(G)$ of $G$ is the ideal
$$
I(G)=(x_ix_j \mid \{x_i,x_j\} \in E(G) ).
$$
Several authors have settled the problem of determining the stabilizing index and the constant for special families of graphs. 
Banerjee proved that $ \reg{I(G)^q}=2q $, for all $ q\geq2 $,   when $ G $ is a gap-free and cricket-free graph (see  \cite{banerjee}). 
Moghimian, Fakhari, and Yassemi answered the question for the family of whiskered graphs (see \cite{Yassemi}). 
Beyarslan, H\`a, and Trung settled the problem for the family of forests and cycles (see \cite{TAI_FOREST_CYCLE}). 
Their results were expanded to the family of unicyclic graphs by Alilooee,  Beyarslan, and Selvaraja (see \cite{REG_UNICYCLIC_GRAPH}).
Moreover, Alilooee and Banerjee determined the stabilizing index and the constant for the family of bipartite graphs with regularity equal to three (see \cite{Alilooee_Banerjee_Arindam_three_bipartite_graphs}).
Jayanthan and Selvaraja settled the problem for the family of very well-covered graphs (see \cite{Jayanthan_Selvaraja_very_well_covered_graphs}). 
Recently, Erey proved that if $ G $ is a gap-free and diamond-free graph, then $ \reg{I(G)^q}=2q $ for all $ q\geq 2 $ (see \cite{Erey_Powers_Edge_Ideals_Linear_Resolutions}).
The approach is focused on the relations between the combinatorics of graphs and algebraic properties of edge ideals. We refer the reader to see \cite{Katzman_Mordechai}, \cite{Tai_Adam_reg_of_hypergraphs}, \cite{Biyikoglu_regularity_of_graphs},   \cite{Hibi_Higashitani_Kimura_Tsuchiya_induced_matchings_regularity}, \cite{Fakhari_Yassemi}, \cite{Woodroofe} and \cite{Norouzi_Fakhari_Yassemi_Reg_Powers_very_well_covered_graphs_Norouzi_Fakhari_Yassemi} for more information on this topic. 
The purpose of this paper is to extend the results of \cite{REG_UNICYCLIC_GRAPH} to the 
family of bicyclic graphs (i.e. a graph with exactly two cycles).

The base case of the family of bicyclic graphs is that of dumbbell graphs. 
A dumbbell graph $ \cnlcm{l} $ is a graph consisting of two cycles $C_n$ and $C_m$ connected with a path  $P_l$, where $ n $, $ m $, and $ l $ are the number of  vertices (see \autoref{example_first_dumbbells}).
For convenience of notation, we define the following function
$$
\xi_3(n) = \begin{cases}
1 \qquad \text{ if } n \equiv 0, 1 \;(\Mod 3),\\
0 \qquad \text{ if } n \equiv 2 \;(\Mod 3).
\end{cases}
$$ 

Here, we describe the basic outline and main results of this paper.

In \hyperref[Preliminaries_section]{Section 1}, we fix some notations and recall known results which are crucial to our approach.

In \hyperref[Regularity_of_the_dumbbell_section]{Section 2}, we use combinatorial techniques to  compute the induced matching number of a dumbbell graph. 
Then,  applying inductive methods, we study the regularity of the edge ideals of dumbbell graphs.
 For a dumbbell graph $ \cnlcm{l} $, we will always assume that $n \;\Mod3 \le m  \;\Mod3 $. The cases $n \equiv 2 \;(\Mod 3)$, $m \equiv 0,1 \; (\Mod 3)$  will have the same results as the cases $n \equiv 0,1 \;(\Mod 3)$, $m \equiv 2 \; (\Mod 3)$. 
Our approach is based on the Lozin transformation (see \cite{LOZIN_TRANSFORMATION} and \cite{LOZIN_TRANS}), and the induced matching number of a dumbbell graph. 
The following results are given in this section:  
\begin{headtheorem}[\autoref{formula_nu_dumbbell}]
	Let $n,m \ge 3$ and $l \ge 1$, then
	$$
	\nu(\cnplcm) = \Big\lfloor \frac{n}{3} \Big\rfloor + \Big\lfloor \frac{m}{3} \Big\rfloor + \Big\lfloor \frac{l-\xi_3(n)-\xi_3(m)+1}{3} \Big\rfloor.
	$$
\end{headtheorem}

\begin{headtheorem}[\autoref{reg_dumbbell}]
	Let $m, n \ge 3$ and $l \ge 1$,
	\begin{enumerate}[(i)]
		\item if $l\equiv 0,1 \;(\Mod 3)$, then
		$$
		\reg{ I(\cnlcm{l}) } = \begin{cases}
		\nu(\cnlcm{l})+2 \qquad \text{\normalfont if } n , m\equiv 2  \;(\Mod 3),\\
		\nu(\cnlcm{l}) +1 \qquad  \text{\normalfont otherwise;}
		\end{cases}
		$$ 
		\item if $l\equiv 2 \;(\Mod3)$, then 
		$$
		\reg{\edge{\cnplcm}}=
		\begin{cases}
		\nu(\cnplcm)+2 \;   & \text{\normalfont if } n\equiv 0,1 \;(\Mod3), \; m\equiv 2 \;(\text{mod}\;3)\\
		\nu(\cnplcm)+1 \;   & \text{\normalfont otherwise}.\\
		\end{cases}
		$$
	\end{enumerate} 
\end{headtheorem}

In \hyperref[reg_bicyclic_graph_section]{Section 3}, for an arbitrary bicyclic graph $G$, we give a combinatorial characterization of $\reg{I(G)}$ in terms of the induced matching number $\nu(G)$.

\begin{headtheorem}[\autoref{full_characterization}]
	Let $G$ be a bicyclic graph with dumbbell $\cnplcm$.
	The following statements hold.
	\begin{enumerate}[(I)]
		\item Let $ n , m\equiv 0,1  \;(\Mod 3)$, then  $\reg{I(G)} = \nu(G) + 1$.
		\item Let $ n \equiv 0,1  \;(\Mod 3)$ and $m\equiv 2  \;(\Mod 3)$, then  
		$$\nu(G) + 1\le\reg{I(G)} \le \nu(G) + 2,$$ 
		and  $\reg{I(G)} = \nu(G) + 2$ if and only if $\nu(G) = \nu(G\setminus \Gamma_G(C_m))$.
		\item Let $n,m \equiv 2 \; (\Mod3)$ and $l \ge 3$, then  $\nu(G) + 1\le\reg{I(G)} \le \nu(G) + 3$.
		Moreover:
		\begin{enumerate}[(i)]
			\item  $\reg{I(G)} = \nu(G) + 3$   if and only if $\nu\left(G\setminus \Gamma_G(C_n\cup C_m)\right) = \nu(G)$.
			\item  $\reg{I(G)} = \nu(G) + 1$  if and only if the following conditions hold:
			\begin{enumerate}[(a)]
				\item $\nu(G)-\nu(G\setminus \Gamma_G(C_n\cup C_m)) > 1$;
				\item $\nu(G) > \nu(G \setminus \Gamma_G(C_n))$;
				\item $\nu(G) > \nu(G \setminus \Gamma_G(C_m))$.
			\end{enumerate}
		\end{enumerate}
		\item Let $n,m \equiv 2 \; (\Mod3)$ and $l \le 2$, then  $\nu(G) + 1\le\reg{I(G)} \le \nu(G) + 2$.
		If $x$ is a vertex on  $P_l$ and  $\LozGraph$ is the Lozin transformation of $G$ with respect to $x$, then  $\reg{I(G)}=\nu(G)+1$  if and only if the following conditions are satisfied: 
		\begin{enumerate}[(a)]
			\item $\nu(\LozGraph)-\nu(\LozGraph\setminus \Gamma_{\LozGraph}(C_n\cup C_m)) > 1$;
			\item $\nu(\LozGraph) > \nu(\LozGraph \setminus \Gamma_{\LozGraph}(C_n))$;
			\item $\nu(\LozGraph) > \nu(\LozGraph \setminus \Gamma_{\LozGraph}(C_m))$.
		\end{enumerate}
	\end{enumerate}
\end{headtheorem}

In \hyperref[Upper_bound_section]{Section 4}, we investigate the asymptotic behavior of regularity of powers of  $ I(\cnlcm{l}) $ when $ l\leq2 $. The approach takes advantage of the notion of even-connectedness and the relations between the induced matching number of graphs and the regularity of the edge ideal.

\begin{headtheorem}[\autoref{powers_lower_bound}]
	Let $\cnplcm$ with $l\le 2$, then
	$$
	\reg{{I(\cnplcm)}^q} = 2q + \reg{I(\cnplcm)} - 2
	$$
	for any $q \ge 1$.
\end{headtheorem} 

For the case $ l\ge3 $, there are immediate examples for which the above theorem does not hold (see  \autoref{Equality_counter_example}).

\section{Preliminaries}\label{Preliminaries_section} 

Let $R=K[x_1,\dots,x_r]$ be the standard graded polynomial ring over a field $K$ and let $\mathbf{m}=(x_1,\ldots,x_r)$ be its maximal homogeneous ideal. For a graded $R$-module $ M $, one can define the Castelnuovo-Mumford regularity in different terms. We recall the definition of the regularity of an $R$-module $ M $ by the minimal free resolution $ M $. The \textit{minimal graded free resolution of $M$} is an exact sequence of the form 
% A minimal free resolution of $M$ is an exact sequence

\begin{center}
$0\rightarrow F_p\rightarrow F_{p-1}\rightarrow \cdots \rightarrow F_0\rightarrow M \rightarrow 0, $
\end{center}
where each $F_i$ is a  graded free $R$-module of the form $F_i=\bigoplus \limits_{\substack{j \in \mathbb{N}}} R(-j)^{\beta_{i,j}(M)} $, each ${\varphi}_i : F_i \rightarrow F_{i-1}$, with $F_{-1} := M$, is a graded homomorphism of degree zero  such that ${\varphi}_{i+1}(F_{i+1}) \subseteq \textbf{m} F_i$ for all $i \geq 0$. The numbers $\beta_{i,j}(M)$ are important invariants, known as the \textit{graded Betti numbers of $M$}. In particular, the number $\beta_i = \sum \limits_{\substack{j \in \mathbb{N}}} \beta_{i,j}(M)$ is called the \textit{i-th Betti number of $M$}  and $\beta_{i,j}(M)$ is the \textit{i-th Betti number  of $M$ of degree $j$}. Note that the minimal free resolution of $M$ is unique up to isomorphism, hence the graded Betti numbers are uniquely determined.

\begin{dfn}
Let $M$ be a finitely generated graded $R$-module. The regularity of $M$ is given by
$$
\reg(M) = \max\lbrace j-i \ | \ \beta_{i,j}(M)\neq 0 \rbrace.
$$
\end{dfn}

\begin{rem}
	Note that, if $I$ is a graded ideal of $R$, then $\reg(R/I)= \reg(I)-1$.
\end{rem}

Let $G= (V,E)$ be a graph with vertex set $V=\lbrace v_1,\dots , v_l\rbrace$. Here, we recall some classes of graphs that we need for this study.
\begin{dfn}
Let $G= (V,E)$ be a graph.
\begin{enumerate}[(i)]
\item $G$ is called a path with $l$ vertices, denoted by $P_l$, if $V=\lbrace v_1,\dots , v_l\rbrace$ and $\{v_i, v_{i+1}\} \in E$ for all $1\leq i\leq l-1$.
\item $G$ is called a cycle with $n$ vertices,  denoted by $C_n$, if $V=\lbrace v_1,\dots , v_n\rbrace$ and $\{v_i, v_{i+1}\} \in E $ for all $1\leq i\leq n-1$ and  $\{v_n, v_1\} \in E$.
\item $G$ is called a dumbbell graph if $G$ contains two cycles $C_n$ and $C_m$ joined by a path $P_l$ of $l$ vertices. We denote it by $\cn \cdot P_l \cdot C_m$. (See \autoref{example_first_dumbbells})
\end{enumerate}
\end{dfn}

 For a vertex $u$ in a graph $G=(V,E)$, let $N_{G}(u)=\lbrace v\in V | \{u, v\} \in E\rbrace$ be the set of \textit{neighbors} of $u$, and set $N_G[u] := N_G(u)\cup \lbrace u\rbrace $. An edge $e$ is \textit{incident} to a vertex $u$ if $u \in e$. The degree of a vertex $u \in V$, denoted by $\deg _G(u)$, is the number of edges incident to $u$. When there is no confusion, we will omit $G$ and write $N(u),N[u]$ and $\deg(u)$.
For an edge $e$ in a graph $G=(V,E)$, we define $G\setminus e$ to be the subgraph of $G$ obtained by deleting $e$ from $E$ (but the vertices are remained). For a subset $W \subseteq V$ of the vertices in $G$, we define $G \setminus W$ to be the subgraph of G deleting the vertices of W and their incident edges. When $W = \lbrace u\rbrace$ consists of a single vertex, we write $G\setminus u$ instead of $G\setminus \lbrace u\rbrace$. For an edge $e = \lbrace u,v\rbrace \in E$, let $N_G[e] = N_G[u] \cup N_G[v]$ and define $G_e$ to be the induced subgraph of $G$ over the vertex set $V \setminus N_G[e]$.

One can think of the vertices of  $G=(V,E)$ as the variables of the polynomial ring $R=K[x_1,\dots,x_r]$ for convenience.
Similarly, the edges of $G$ can be considered as square free monomials of degree two. 
By abuse of notation, we use $e$ to refer to both the edge $e=\{x_i,x_j\}$ and the monomial $e=x_ix_j \in I(G)$.
%\begin{dfn}
%	The \textit{edge ideal} of the graph $G=(V,E)$ is the square free monomial ideal
%	$$
%	I(G) = (x_ix_j \ | \ e_{i,j} \in E).
%	$$
%\end{dfn}

\medskip

Let $G=(V,E)$ be a graph and  $W \subseteq V$. The \textit{induced subgraph of $G$ on $W$}, denoted by $G[W]$, is the graph with vertex set $W$ and edge set $\{ e \in E \ | \ e \subseteq W \}$.

\begin{dfn}
	Let $G=(V,E)$ be a graph.
	
	\noindent A collection $C$ of edges of $G$ is called a \textit{matching} if the edges in $C$ are pairwise disjoint. The maximum size of a matching in $G$ is called its \textit{matching number}, which is denoted by $\normalfont\text{match}(G)$.
	
	\noindent A collection $C$ of edges of $G$ is called an \textit{induced matching} if $C$ is a matching, and $C$ consists of all edges of the induced subgraph $G \big[ \bigcup \limits_{\substack{e \in C}} e \big]$ of $G$. The maximum size of an induced matching in $G$ is called its \textit{induced matching number} and it is denoted by $\nu(G)$.
\end{dfn}

\begin{rem}(\cite[Remark 2.12]{TAI_FOREST_CYCLE})\label{nu_path}
	Let $P_l$ be a path of $l$ vertices, then we have $$\nu(P_l)=\lfloor \frac{l+1}{3} \rfloor$$
\end{rem}

\begin{rem} \label{rem_struct_match_circle}
	(\cite[Remark 2.13]{TAI_FOREST_CYCLE})
	Let $C_n$ be a cycle of $n$ vertices, then we have 
	$$
	\nu(C_n)=\lfloor \frac{n}{3} \rfloor.
	$$
%	A maximal induced matching of $\cn$ is completely determined by just choosing a first edge, and then we go (for instance) in clockwise direction by taking the third consecutive edge after the last one chosen.
	
%	Thus, we shall use $r = n \;\Mod3$ to give a specific characterization of the structure of the maximal induced matching. 
	Depending on $r = n \;\Mod3$ we can assume the following:
	\begin{enumerate}[(i)]
		\item when $r = 0$, there exists a maximal induced matching of $C_n$ that does not contain the edges $x_1x_2$ and $x_1x_n$;
		\item when $r = 1$, there exists a maximal induced matching of $C_n$ that does not contain the edges $x_1x_2$, $x_1x_n$ and $x_{n-1}x_n$;
		\item when $r = 2$, there exists a maximal induced matching of $C_n$ that does not contain the edges $x_1x_2$, $x_2x_3$, $x_1x_n$ and $x_{n-1}x_n$.
	\end{enumerate}
\end{rem}

\begin{theorem}\cite[Lemma 3.1, Theorems 3.4 and 3.5]{Tai_survey}
	\label{short exact sequence}
	Let $G=(V,E)$ be a graph.
	\begin{enumerate}[(i)]
		\item If $H$ is an induced subgraph of $G$, then $\reg {\edge H} \leq \reg{\edge G};
		$
		\item Let $x\in V$, then 
		$$
		\reg{\edge G}\leq \max \lbrace \reg {\edge{G\setminus x}}, \reg {\edge {G \setminus N[x]}}+1\rbrace;
		$$
		\item Let $e\in E$, then 
		$$
		\reg{\edge G}\leq \max \lbrace2,\reg{\edge {G\setminus e}}, \reg{\edge {G_e}}+1 \rbrace.
		$$
	\end{enumerate}
\end{theorem}

Now we recall the concept of even-connection introduced by Banerjee in \cite{banerjee}.

\begin{dfn}[\cite{banerjee}]
	Let $G = (V,E)$ be a graph with edge ideal $I = \edge G$. Two 
	vertices $x_i$ and $x_j$ in $G$ are called \textit{even-connected} with 
	respect to an $s$-fold product $M = e_1\cdots e_s$, 
	where $e_1,\dots ,e_s$ are edges in $G$, if there is a path $p_0,
	\dots,p_{2l+1}$, for some $l \geq 1$, in $G$ such that the following 
	conditions hold:
	\begin{enumerate}[(i)]
		\item $p_0  =  x_i $ and $p_{2l+1} =  x_j$;
		\item for all $0\leq j\leq l-1, \lbrace p_{2j+1}, p_{2j+2}\rbrace= e_i$ for some $i$;
		\item for all $i$, $\big{|} \lbrace j \ | \ \lbrace p_{2j+1}, p_{2j+2}\rbrace= e_i \rbrace\big{|} \leq \big{|} \lbrace t\ |\  e_t=e_i\rbrace \big{|}$.
	\end{enumerate}
\end{dfn}

%\begin{dfn}[\cite{banerjee}]
%	The edges $e_1=x_{1,1}x_{1,2},\,\dots \,, e_q=x_{q,1}x_{q,2}$ are in an \textit{even-connected position}, if for all $1\leq i\leq q-1$, the vertex $x_{i,2}$ is connected to the vertex $x_{i+1,1}$ and there exist $u\in N(e_1)$ and $v\in N(e_q)$ such that $u$ and $v$ are even-connected with respect to $e_1 \cdots e_q$.
%\end{dfn}

%For the edge ideal $I = I(G)$ of some $G = (V,E)$ and an integer $s \geq  1$, the following holds.
\begin{theorem}\cite[Theorems 6.1 and 6.5]{banerjee}\label{G'}
	 Let $M = e_1e_2\cdots e_s$ be a minimal generator of $I^s$. Then $(I^{s+1} \colon M)$ is minimally generated by monomials of degree $2$, and $uv$ ($u$ and $v$ may be the same) is a minimal generator of $(I^{s+1} \colon M)$ if and only if either $\lbrace u,v\rbrace\in E$  or $u$ and $v$ are even-connected with respect to $M$.
\end{theorem}

\begin{rem}\cite[Lemma 6.11]{banerjee}
	\label{newGraphBanerjee}
	Let ${(I^{s+1} \colon M)}^{\text{pol}}$ be the polarization of the ideal $(I^{s+1} \colon M)$ (see e.g. \cite[\S 1.6]{HERZOG_HIBI}).
	From the previous theorem we can construct a graph $G^\prime$ whose edge ideal is given by ${(I^{s+1} \colon M)}^{\text{pol}}$.
	The new graph $G^\prime$ is given by:
	\begin{enumerate}[(i)]
		\item All the vertices and edges of $G$.
		\item Any two vertices $u,v$, $u \neq v$ that are even-connected with respect to $M$ are connected by an edge in $G^\prime$.
		\item For every vertex $u$ which is even-connected to itself with respect to $M$, there is a new vertex $u^\prime$ which is connected to $u$ by an edge and not connected to any other vertex (so $uu^\prime$ is a whisker).
	\end{enumerate}
\end{rem}

%The key to our treatment of the dumbbell graphs is to calculate equal upper bound and lower bound for the regularity of the edge ideal. The following theorems are useful in our proofs.
\begin{theorem}\cite[Theorem 5.2]{banerjee}\label{banerjee}
	Let $G$ be a graph and  $\lbrace m_1,\dots,m_r\rbrace$ be the set of minimal monomial generators of $\edge {G}^q$ for all $q\geq 1$, then 
	$$
	\reg{\edge G^{q+1}}\leq \max \lbrace \reg{(\edge G^q \colon m_l)} +2q, 1\leq l\leq r, \reg {\edge{G}^q} \rbrace. 
	$$
\end{theorem}

Here by, we recall  a result by Kalai and Meshulam on the regularity of monomial ideals.
\begin{theorem}\cite{Kalai_Gil_Meshulam_Roy}\label{Kalai-Meshulam}
	Let $I_1,\dots , I_s$ be monomial ideals in $R$, then
	$$ \reg{ \left( R \Big{/} \sum_{i=1}^{s} I_i \right)} \leq  \sum_{i=1}^{s} \reg (R/I_i).$$
\end{theorem}

The regularity of the edge ideal of a forest was first computed by Zheng in \cite[Theorem 2.18]{ZHENG_FORESTS}.

\begin{theorem}\cite[Theorem 2.18]{ZHENG_FORESTS}
	\label{reg_Forests}
	Let $G$ be a forest, then 
	$$
	\reg{I(G)} = \nu(G) + 1.
	$$
\end{theorem}

In \cite{Katzman_Mordechai} Katzman first noticed that the previous equality is a lower bound for general graphs.

\begin{theorem}\cite[Corollary 1.2]{Katzman_Mordechai}
	\label{lower_bound_Katzman}
	Let $G$ be a graph, then 
	$$
	\reg{I(G)} \ge \nu(G) + 1.
	$$
\end{theorem}

The decycling number of a graph is an important combinatorial invariant which can be used to obtain an upper bound for the regularity of the edge ideal of a graph.

\begin{dfn}
	For a graph $G$ and $D \subset V(G)$, if $G \setminus D$ is acyclic, i.e. contains no induced cycle, then $D$ is said to be a decycling set of $G$. 
	The size of a smallest decycling set of $G$ is called the decycling number of $G$ and denoted by $\nabla(G)$.
\end{dfn}

\begin{theorem}
	\label{Upper_bound_reg_decyc}
	\cite[Theorem 4.11]{LOZIN_TRANS}
	Let $G$ be a graph, then 
	$$
	\reg{I(G)} \le \nu(G)  + \nabla(G) + 1.
	$$
\end{theorem}

In \cite{TAI_FOREST_CYCLE} Beyarslan, H\`{a} and Trung provided a formula for the regularity of the powers of edge ideals of  forests and cycles in terms of the induced matching number.

\begin{theorem}\cite[Theorem 4.7]{TAI_FOREST_CYCLE}\label{TAI_REG_OF_FOREST_CYCLE}
		Let $G$ be a forest, then
		$$
		\reg{I(G)^q} = 2q + \nu(G) -1.
		$$
		for all $q\ge 1$.
\end{theorem}

\begin{theorem}\label{regularity Cn}\cite[Theorem  5.2]{TAI_FOREST_CYCLE}. Let   $C_n$ be a cycle with $ n $ vertices, then 
	$$
	\reg I(C_n)= \begin{cases}
	\nu(C_n) +1 \qquad \text{ if } n \equiv 0, 1 \;(\Mod 3),\\
	\nu(C_n) +2 \qquad \text{ if } n \equiv 2 \;(\Mod 3),
	\end{cases}
	$$ 
	where $\nu(C_n) = \lfloor \dfrac{n}{3}\rfloor $ denote the induced matching number of $\cn$.
	Moreover,
	$$
	\reg I(C_n)^q = 2q+\nu(C_n) -1.
	$$
	and for all $q\ge 2$.
\end{theorem}

In addition, the authors of \cite{TAI_FOREST_CYCLE} provided a lower bound for the regularity of the powers of the edge ideal of an arbitrary graph, and an upper bound for the regularity of the edge ideal of a graph containing a Hamiltonian path.

\begin{theorem}\cite[Theorem 4.5]{TAI_FOREST_CYCLE}\label{TAI_REG_LOWER_BOUND}
	Let $ G $ be a graph and let $ \nu(G) $ denote its induced matching number. Then, for all $ q\geq1 $, we have
	$$ \reg{ \edge{G}^q}\geq2q +\nu(G)-1 $$
\end{theorem}

\begin{theorem}\label{Hamiltonian}\cite[Theorem  3.1]{TAI_FOREST_CYCLE}
Let $G$ be a graph on $n$ vertices. Assume $G$ contains a Hamiltonian path, then
$$
\reg I(G)\leq \lfloor \dfrac{n+1}{3}\rfloor+1
$$
\end{theorem}

\section{Regularity and induced matching number of a dumbbell graph}\label{Regularity_of_the_dumbbell_section}

In this section we compute the induced matching number of a dumbbell graph and the regularity of its edge ideal.
Recall that $\cnplcm$  denotes the graph constructed by joining two cycles  $C_n$ and $C_m$ via a path $P_l$.
In this section, we denote the vertices of $C_n$, $C_m$ and $P_l$ by $\{x_1, \ldots, x_n\}$, $\{ y_1, \ldots, y_m\}$ and $\{z_1, \ldots, z_{l} \}$, respectively.
We make the identifications $x_1=z_1$ and $y_1=z_l$.

\begin{exmp}
	\label{example_first_dumbbells}
	Two base cases when $l=2$ and $l=1$ are the following: 
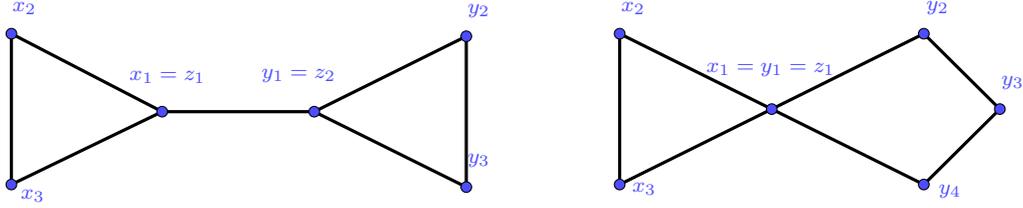
\begin{figure}[H]
		\definecolor{ududff}{rgb}{0.30196078431372547,0.30196078431372547,1.}
		\begin{minipage}[t]{\fboxrule}% 
			\hspace*{-7cm}{
				\begin{tikzpicture}[line cap=round,line join=round,>=triangle 45,x=1cm,y=1cm]
				\clip(-1.5,2.7) rectangle (12.5,5.5);
				\draw [line width=1.2pt] (-0.9973727686154237,5.)-- (-0.9973727686154237,3.);
				\draw [line width=1.2pt] (-0.9973727686154237,3.)-- (0.9855556692596585,3.9658568757501644);
				\draw [line width=1.2pt] (0.9855556692596585,3.9658568757501644)-- (-0.9973727686154237,5.);
				\draw [line width=1.2pt] (0.9855556692596585,3.9658568757501644)-- (2.9855556692596563,3.9658568757501644);
				\draw [line width=1.2pt] (2.9855556692596563,3.9658568757501644)-- (4.9855556692596465,4.965856875750165);
				\draw [line width=1.2pt] (4.9855556692596465,4.965856875750165)-- (4.9855556692596465,2.9658568757501644);
				\draw [line width=1.2pt] (4.9855556692596465,2.9658568757501644)-- (2.9855556692596563,3.9658568757501644);
				\draw [line width=1.2pt] (7.0026272313845705,5.)-- (7.0026272313845705,3.);
				\draw [line width=1.2pt] (7.0026272313845705,3.)-- (9.002627231384558,4.);
				\draw [line width=1.2pt] (9.002627231384558,4.)-- (7.0026272313845705,5.);
				\draw [line width=1.2pt] (9.002627231384558,4.)-- (11.002627231384558,5.);
				\draw [line width=1.2pt] (11.002627231384558,5.)-- (12.002627231384558,4.);
				\draw [line width=1.2pt] (12.002627231384558,4.)-- (11.002627231384558,3.);
				\draw [line width=1.2pt] (11.002627231384558,3.)-- (9.002627231384558,4.);
				\begin{scriptsize}
				\draw [fill=ududff] (-0.9973727686154237,5.) circle (2.0pt);
				\draw[color=ududff] (-0.8254093679721619,5.343220543689411) node {$x_2$};
				\draw [fill=ududff] (-0.9973727686154237,3.) circle (2.0pt);
				\draw[color=ududff] (-0.7157344243058552,2.85209254863818) node {$x_3$};
				\draw [fill=ududff] (0.9855556692596585,3.9658568757501644) circle (2.0pt);
				\draw[color=ududff] (1.0573438316327706,4.429262679803522) node {$x_1=z_1$};
				\draw [fill=ududff] (2.9855556692596563,3.9658568757501644) circle (2.0pt);
				\draw[color=ududff] (2.783118099294758,4.44754183708124) node {$y_1=z_2$};
				\draw [fill=ududff] (4.9855556692596465,4.965856875750165) circle (2.0pt);
				\draw[color=ududff] (5.151875061841555,5.306662229133976) node {$y_2$};
				\draw [fill=ududff] (4.9855556692596465,2.9658568757501644) circle (2.0pt);
				\draw[color=ududff] (5.151875061841555,3.314234085862738) node {$y_3$};
				\draw [fill=ududff] (7.0026272313845705,5.) circle (2.0pt);
				\draw[color=ududff] (7.18086151966823,5.343220543689411) node {$x_2$};
				\draw [fill=ududff] (7.0026272313845705,3.) circle (2.0pt);
				\draw[color=ududff] (7.327094777889973,2.9400467266969713) node {$x_3$};
				\draw [fill=ududff] (9.002627231384558,4.) circle (2.0pt);
				\draw[color=ududff] (8.989427360107396,4.521729196247007) node {$x_1=y_1=z_1$};
				\draw [fill=ududff] (11.002627231384558,5.) circle (2.0pt);
				\draw[color=ududff] (11.183996963488426,5.343220543689411) node {$y_2$};
				\draw [fill=ududff] (12.002627231384558,4.) circle (2.0pt);
				\draw[color=ududff] (12.171071456485185,4.356146050692651) node {$y_3$};
				\draw [fill=ududff] (11.002627231384558,3.) circle (2.0pt);
				\draw[color=ududff] (11.348509378987885,2.9120926257529467) node {$y_4$};
				\end{scriptsize}
				\end{tikzpicture}}
		\end{minipage}		
		\caption{The graphs $C_3 \cdot P_2 \cdot C_3$ and $C_3\cdot P_1 \cdot C_4$.}
\end{figure}
\end{exmp}

\begin{nota}
	Let $\xi_3$ be the function defined as below
	$$
	\xi_3(n) = \begin{cases}
		1 \qquad \text{ if } n \equiv 0, 1 \;(\Mod 3),\\
		0 \qquad \text{ if } n \equiv 2 \;(\Mod 3).
	\end{cases}
	$$ 
\end{nota}

Let $C_n\cdot P_l$ be the graph given by connecting the path $P_l$ to the cycle $C_n$.  For instance, the graph $C_3\cdot P_3$ can be illustrated as the following: 
	\begin{center}
		\definecolor{ududff}{rgb}{0.30196078431372547,0.30196078431372547,1.}
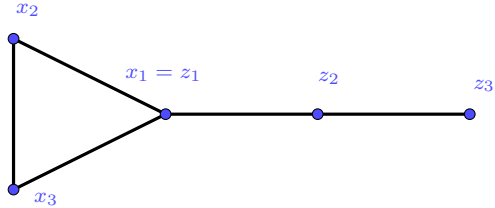
			\begin{tikzpicture}[line cap=round,line join=round,>=triangle 45,x=1.0cm,y=1.0cm]
			\clip(2.5,1.6) rectangle (9.7,4.5);
			\draw [line width=1.2pt] (3.,4.)-- (3.,2.);
			\draw [line width=1.2pt] (3.,2.)-- (5.,3.);
			\draw [line width=1.2pt] (5.,3.)-- (3.,4.);
			\draw [line width=1.2pt] (5.,3.)-- (7.,3.);
			\draw [line width=1.2pt] (7.,3.)-- (9.,3.);
			\begin{scriptsize}
			\draw [fill=ududff] (3.,4.) circle (2.0pt);
			\draw[color=ududff] (3.19,4.38) node {$x_2$};
			\draw [fill=ududff] (3.,2.) circle (2.0pt);
			\draw[color=ududff] (3.43,1.88) node {$x_3$};
			\draw [fill=ududff] (5.,3.) circle (2.0pt);
			\draw[color=ududff] (4.97,3.52) node {$x_1=z_1$};
			\draw [fill=ududff] (7.,3.) circle (2.0pt);
			\draw[color=ududff] (7.15,3.48) node {$z_2$};
			\draw [fill=ududff] (9.,3.) circle (2.0pt);
			\draw[color=ududff] (9.19,3.38) node {$z_3$};
			\end{scriptsize}
			\end{tikzpicture}
	\end{center}

\begin{prop}
	\label{nu_path_cycle}
	Let $n \ge 3$ and $l \ge 1$, then
	$$
	\nu(C_n \cdot P_l) = \Big\lfloor \frac{n}{3} \Big\rfloor + \Big\lfloor \frac{l - \xi_3(n) + 1}{3} \Big\rfloor.
	$$
		\begin{proof}
		\underline{Case 1:}
		From \autoref{rem_struct_match_circle}, in the case $n \equiv 2 \,(\Mod 3 )$ we have that in clockwise and anticlockwise directions the two consecutive edges to the vertex $x_1$ are not chosen in a maximal induced matching of $C_n$. 
		Then, we can choose the edges in $P_l$ without any constraint coming from the maximal induced matching chosen in $C_n$, and so we have $\nu(C_n\cdot P_l) = \lfloor \frac{n}{3} \rfloor + \lfloor \frac{l+1}{3} \rfloor$.
		
		\underline{Case 2:} It remain to consider the case $\xi_3(n)=1$, i.e., $n \equiv 0, 1 \; (\Mod  3)$.
		Let $\matching$ be an induced matching of maximal size in $G$.
		We analyze separately the two cases of whether $z_1z_2$ (the edge adjacent to the cycle $C_n$) is in $\matching$ or not.
		
		Suppose  $z_1z_2$ is not an edge of $\matching$.
		Then $\matching$ can be considered as the union of a maximal matching of $C_n$ as introduced in \autoref{rem_struct_match_circle} and a maximal matching of the path $P_l \setminus z_1$.
		Thus $\lvert \matching \rvert = \nu(C_n) + \nu(P_{l-1})=\lfloor \frac{n}{3} \rfloor +  \lfloor \frac{(l-1)+1}{3}\rfloor$.
		
		If $z_1z_2 \in \matching$, then none of the edges incident to the vertices in $N_{C_n}[x_1]=\{x_1,x_2,x_n\}$ are in ${\matching\mid}_{ C_n}:=\{e \in \matching \mid e \in C_n \}$. 
		Hence $\lvert {\matching\mid}_{ C_n} \rvert = \nu(P_{n-3})$, and since $n \equiv 0, 1 \; (\Mod  3)$ then it follows $\lvert {\matching\mid}_{ C_n} \rvert= \lfloor \frac{n-2}{3} \rfloor = \lfloor \frac{n}{3}\rfloor - 1 $.
		From $z_1z_2 \in \matching$ we get $\lvert {\matching\mid}_{P_l} \rvert = \nu(P_l)=\lfloor \frac{l+1}{3} \rfloor$.
		So, by joining both computations we get
		$
		\lvert\matching \rvert = 	\lfloor \frac{n}{3}\rfloor - 1 + \lfloor \frac{l+1}{3} \rfloor = \lfloor \frac{n}{3} \rfloor  + \lfloor \frac{l-2}{3} \rfloor. 
		$
		
		Therefore, we obtain that $\nu(C_n\cdot P_l)=\lfloor \frac{n}{3} \rfloor +  \lfloor \frac{(l-1)+1}{3}\rfloor$.		
	\end{proof}
\end{prop}

\begin{theorem}
	\label{formula_nu_dumbbell}
	Let $n,m \ge 3$ and $l \ge 1$, then 
	$$
	\nu(\cnplcm) = \Big\lfloor \frac{n}{3} \Big\rfloor + \Big\lfloor \frac{m}{3} \Big\rfloor + \Big\lfloor \frac{l-\xi_3(n)-\xi_3(m)+1}{3} \Big\rfloor.
	$$
	\begin{proof}
	We use the same argument as in \autoref{nu_path_cycle}.
	By \autoref{rem_struct_match_circle} we have that when either $n \equiv 2 \; (\Mod  3)$ or $m \equiv 2 \; (\Mod 3)$, then the maximal induced matching in $C_n$ or in $C_m$ does not affect the way we choose edges in the path $P_l$.
	
	In the case $n \equiv 0,1 \; (\Mod 3)$ we can choose a maximal induced matching that does not use the edge connected to the cycle $C_n$, which is the same as saying that we are not going to use one extreme vertex of the path $P_l$. 
	Similarly, when $m \equiv 0,1 \; (\Mod 3)$ we can drop the other extreme vertex. 
\end{proof}
\end{theorem}

The aim of the  rest of this section is to explicitly compute the regularity of $ \edge{\cnlcm{l}} $ in term of the induced matching number.  We divide it into three subsections depending on the value of $l\; \text{mod}\; 3$.
The base of our computations is given by the following proposition.

\begin{prop}
	\label{induct_skel_lozin}
	Let $n, m \ge 3$ and $l \ge 1$, then 
	$$
	\reg{I(\cnplcm)} -  \nu(\cnplcm) = \reg{I(C_n \cdot P_{l+3} \cdot C_m)} - \nu(C_n \cdot P_{l+3} \cdot C_m).
	$$
		\begin{proof}	
		From the formula obtained in \autoref{formula_nu_dumbbell} or \cite[Lemma 1]{LOZIN_TRANSFORMATION}, we have the equality
		$$
		\nu(C_n\cdot P_{l+3} \cdot C_m) = \nu(\cnplcm) + 1.
		$$
		We can apply the Lozin transformation (see e.g. \cite{LOZIN_TRANSFORMATION}, \cite{LOZIN_TRANS}) to any of the vertices in the bridge $P_l$, then from \cite[Theorem 1.1]{LOZIN_TRANS} we have 
		$$
		\reg{I(C_n \cdot P_{l+3} \cdot C_m)} = \reg{I(\cnplcm)} + 1.
		$$
		Thus, the statement of the proposition follows by subtracting these equalities.  
	\end{proof}
\end{prop}

From the previous proposition,  it follows that we only need to consider the cases $l = 1$, $l = 2$ and $l = 3$.
We treat each case in a separate subsection.
In the following theorem we compute the regularity of the edge ideal of the dumbbell $\cnplcm$.

\begin{theorem}
	\label{reg_dumbbell}
	Let $m, n \ge 3$ and $l \ge 1$, then
	\begin{enumerate}[(i)]
		\item if $l\equiv 0,1 \;(\Mod 3)$, then
		$$
		\reg{ I(\cnlcm{l}) } = \begin{cases}
		\nu(\cnlcm{l})+2 \qquad \text{\normalfont if } n , m\equiv 2  \;(\Mod 3),\\
		\nu(\cnlcm{l}) +1 \qquad  \text{\normalfont otherwise;}
		\end{cases}
		$$ 
		\item  if $l\equiv 2 \;(\Mod 3)$, then
		$$
		\reg{\edge{\cnplcm}}=
		\begin{cases}
		\nu(\cnplcm)+2 \;  & n\equiv 0,1 \;(\Mod 3), \; m\equiv 2 \;(\text{mod}\;3);\\
		\nu(\cnplcm)+1 \;   & \text{\normalfont otherwise}.\\
		\end{cases}
		$$
	\end{enumerate} 
	\begin{proof}
		Follows from \autoref{induct_skel_lozin}, and \autoref{reg of cn p1 cm}, \autoref{formula_reg_dumbbell}, and \autoref{reg of cn p3 cm}. 
	\end{proof}
\end{theorem} 

The basic approach in the next three subsections is to obtain lower and upper bounds that coincide. 

\subsection{The case $l=1$}\hspace{\fill} \\
Throughout this subsection, we consider the dumbbell graph $ \cnlcm{1} $.

\begin{prop}
	\label{ineq_l_=1}
	Let $n,m \ge 3$, then 
	$$
	\reg{I(C_n\cdot P_1 \cdot C_m)} \le \max \Big\{ \Big\lfloor \frac{n}{3} \Big\rfloor + \Big\lfloor \frac{m}{3} \Big\rfloor + 1,  \Big\lfloor \frac{n-2}{3} \Big\rfloor + \Big\lfloor \frac{m-2}{3} \Big\rfloor +  2 \Big\}.
	$$
	Moreover, $\reg{I(C_n\cdot P_1 \cdot C_m)}$ is equal to one of these terms.
	\begin{proof}
		We  use \cite[Lemma 3.2]{DAO_HUNEKE}, that gives an improved version of the exact sequence coming from deleting the vertex $z_1$.
		We have  
		$$
		\reg{I(C_n\cdot P_1 \cdot C_m)} \in \Big\{ \reg{I\big((C_n\cdot P_1 \cdot C_m) \setminus z_1\big)}, \reg{I\big((C_n\cdot P_1 \cdot C_m) \setminus N[z_1]\big)} + 1
		\Big\}.
		$$
		Since $(C_n\cdot P_1 \cdot C_m) \setminus z_1=P_{n-1} \cup P_{m-1}$ and $(C_n\cdot P_1 \cdot C_m) \setminus N[z_1]=P_{n-3} \cup P_{m-3}$,  we get the result by applying \autoref{reg_Forests}.
	\end{proof}
\end{prop}

\begin{theorem}\label{reg of cn p1 cm}
	Let $n,m\ge 3$, then
	$$
	\reg{\edge{C_n \cdot P_1 \cdot C_m}}=
	\begin{cases}
	\nu(C_n \cdot P_1 \cdot C_m)+2 \;  & \text{\normalfont if } n\equiv 2 \;(\Mod 3), \; m\equiv 2 \;(\Mod 3);\\
	
	\nu(C_n \cdot P_1 \cdot C_m)+1 \;   & \text{\normalfont otherwise}.\\
	\end{cases}
	$$
	\begin{proof}
		Suppose $n\equiv 2 \;(\text{mod}\;3) $ and $ \; m\equiv 2 \;(\text{mod}\;3)$. Since
		$ \lfloor \frac{k-2}{3} \rfloor=\lfloor \frac{k}{3} \rfloor $ when $ k\equiv 2 \;(\text{mod}\;3) $, we have  
		$$ \max\{ \lfloor \frac{n}{3} \rfloor + \lfloor \frac{m}{3} \rfloor + 1,  \lfloor \frac{n-2}{3} \rfloor + \lfloor \frac{m-2}{3} \rfloor +  2 \}=
		\lfloor \frac{n}{3} \rfloor + \lfloor \frac{m}{3} \rfloor +  2.  $$
		Thus \autoref{ineq_l_=1} yields 
		\begin{equation}\label{eq_3}
			\reg{\edge{C_n \cdot P_1 \cdot C_m}}\leq \lfloor \frac{n}{3} \rfloor + \lfloor \frac{m}{3} \rfloor +  2.
		\end{equation}
		Consider the induced subgraph $H=(C_n \cdot P_1 \cdot C_m) \setminus \{x_n\}$ where $ x_n $ is in $ C_n $ and it is incident to $ x_1 $ (e.g. see $ x_3 $ in \autoref{example_first_dumbbells}). 
		In fact,  $H$ is the graph given by joining $C_m$ and a path  $P_{n-1}$, that is, $H=C_m\cdot P_{n-1}$. 
		Now from    \autoref{nu_path_cycle}, we have that $\nu(H)=\lfloor \frac{n}{3} \rfloor + \lfloor \frac{m}{3} \rfloor $.  
		By \autoref{short exact sequence} $ (i) $, we get $\reg{I(C_n \cdot P_1 \cdot C_m)} \ge \reg{I(H)}$. 
		From \cite[Theorem 1.1]{REG_UNICYCLIC_GRAPH}, we have $\reg{I(H)}=\nu(H)+2$. 
		Therefore, the equality holds in \autoref{eq_3}. 
		The proof of this part is complete since \autoref{formula_nu_dumbbell} yields $ \nu(\cnlcm{1})=\lfloor \frac{n}{3} \rfloor + \lfloor \frac{m}{3} \rfloor  $.
		
		For any case distinct to $n\equiv 2 \;(\text{mod}\;3) $ and $ \; m\equiv 2 \;(\text{mod}\;3)$, we have
		$$ \max\{ \lfloor \frac{n}{3} \rfloor + \lfloor \frac{m}{3} \rfloor + 1,  \lfloor \frac{n-2}{3} \rfloor + \lfloor \frac{m-2}{3} \rfloor +  2 \}=
		\lfloor \frac{n}{3} \rfloor + \lfloor \frac{m}{3} \rfloor +  1.  $$
		Therefore, from \autoref{ineq_l_=1}, we have
		\begin{equation}\label{eq_4}
			\reg{	\edge{C_n \cdot P_1 \cdot C_m}}\leq\lfloor \frac{n}{3} \rfloor + \lfloor \frac{m}{3} \rfloor +1.
		\end{equation}
		From \autoref{formula_nu_dumbbell}, we have $ \nu(C_n \cdot P_1 \cdot C_m)= \lfloor \frac{n}{3} \rfloor + \lfloor \frac{m}{3} \rfloor $.
		Moreover, \autoref{lower_bound_Katzman} gives $\reg{	\edge{C_n \cdot P_1 \cdot C_m}} \ge \nu(C_n \cdot P_1 \cdot C_m) + 1$. Thus, the equality in \autoref{eq_4} holds.
		Therefore the proof is complete. 		
	\end{proof}
\end{theorem}

\subsection{The case $l=2$} \hspace{\fill} \\
Throughout this subsection, we consider the dumbbell graph $\cnlcm{2}$. 
\begin{rem}
	\label{remark_reg_cycle_formula}
	The regularity of  $\edge{\cn}$ is given in \autoref{regularity Cn}.	For simplicity of notation, we  use the equivalent formula $\reg{I(\cn)} = \lfloor \frac{n-2}{3}\rfloor + 2$.
\end{rem}

\begin{prop}
	\label{prop_lower_upper_bound_reg}
	Let $n,m\ge 3$, then 
	\begin{equation}\label{eq_2}
		\nu(\cncm) \le \reg(\frac{R}{I(\cncm)}) \le \lfloor \frac{n-2}{3} \rfloor + \lfloor \frac{m-2}{3} \rfloor +  2.
	\end{equation}
\begin{proof}
	We only need to prove the inequality on the right since the lower bound is given due to \autoref{lower_bound_Katzman} and $ \reg{(J)}-1=\reg{(\frac{R}{J})} $ for any ideal of $J\subset R$. In the original graph $\cncm$ we shall remove the edge  that connects the two cycles $\cn$ and $\cm$. .
	The set of vertices of $\cn$ and $\cm$ are given respectively by $\{x_1, \ldots, x_n\}$ and $\{ y_1, \ldots, y_m	\}$, and we assume that the edge $e = x_1y_1$ is the bridge between the two cycles. Also, we denote by $\cn \cup \cm$ the resulting graph given as the disjoint union of the two cycles $\cn$ and $\cm$. 
	Thus \autoref{short exact sequence}$(iii)$ yields the inequality
	$$
	\reg{\left(\frac{R}{I(\cncm)}\right)} \le \max\Big\{ \reg{\left(\frac{R}{I(\cn \cup \cm):e}\right)} + 1, \reg{\left(\frac{R}{I(\cn \cup \cm)}\right)} \Big\}.
	$$
	
	From \cite[Lemma 3.2]{HOA_TAM} we have that the regularity of the two disjoint cycles $\cn \cup \cm$ is given by 
	$$
	\reg{\left(\frac{R}{I(\cn \cup \cm)}\right)} = \reg{\left(\frac{R}{I(\cn)}\right)} + \reg{\left(\frac{R}{I(\cm)}\right)},
	$$
	and using \autoref{remark_reg_cycle_formula} we get the equality 
	$$
	\reg{\left(\frac{R}{I(\cn \cup \cm)}\right)} = \Big\lfloor \frac{n-2}{3} \Big\rfloor + \Big\lfloor \frac{m-2}{3} \Big\rfloor + 2.
	$$
	Consider the graph $H= \{x_2,x_n\} \cup P_{n-3} \cup \{y_2,y_m\} \cup P_{m-3}$, where $\{x_2,x_n\}$ and $\{y_2,y_m\}$ are incident vertices of graph $ \cncm $ to $ x_1 $ and $ y_1 $ respectively (see \autoref{example_first_dumbbells}). Moreover,  $ P_{n-3} $ is the path with vertices $ x_3,\ldots,x_{n-1} $ and $ P_{m-3} $ is the path with vertices $ y_3,\ldots,y_{m-1} $. It is easy to see that $ \reg{I(H)}=\reg{I(\cn \cup \cm) : e} $. 
	Hence from  \autoref{nu_path}, \autoref{Kalai-Meshulam} and again \cite[Lemma 3.2]{HOA_TAM}  we get
	$$
	\reg{\left(\frac{R}{I(\cn \cup \cm):e}\right)} + 1 = \Big\lfloor \frac{n-2}{3} \Big\rfloor + \Big\lfloor \frac{m-2}{3} \Big\rfloor + 1,
	$$		
	This proves the proposition. 
\end{proof}
\end{prop}

As a result  of the previous proposition, we can prove the following corollary.

\begin{cor}
	\label{several_cases_reg}
	If $n\equiv 0, 1 \;(\Mod 3)$ and $m\equiv 0, 1 \;(\Mod 3)$, then
	$$
	\reg{\left(\frac{R}{I(\cncm)}\right)} = \nu(\cncm) = \Big\lfloor \frac{n}{3} \Big\rfloor + \Big\lfloor \frac{m}{3} \Big\rfloor
	$$
	\begin{proof}
		We note that $\lfloor \frac{k}{3} \rfloor = \lfloor \frac{k-2}{3}\rfloor+1 $ when $k \equiv 0,1 \;(\Mod 3)$.
		From  \autoref{formula_nu_dumbbell},  in \autoref{eq_2} the lower and upper bound coincide for these cases. 
		So the equality is established.
	\end{proof}
\end{cor}

Now we have only three more cases left to deal with, i.e., the case $n \equiv 0 \; (\Mod  3), \; m \equiv 2 \;(\Mod  3)$, the case $n \equiv 1 \; (\Mod  3), \; m \equiv 2 \;(\Mod  3)$, and the case $n \equiv 2 \; (\Mod  3), \; m \equiv 2 \;(\Mod  3)$.

\begin{lem}
	\label{reg_case_two_two}
	 If $n \equiv 2 \; (\Mod  3)$ and $m \equiv 2 \;(\Mod  3)$, then 
	 $$
	 \reg{\left(\frac{R}{I(\cncm)}\right)} = \nu(\cncm) = \Big\lfloor \frac{n}{3} \Big\rfloor + \Big\lfloor \frac{m}{3} \Big\rfloor + 1. 
	 $$
	  \begin{proof}
	 	We shall divide the graph into three subgraphs $H_1$, $H_2$ and $H_3$. 
	 	We make $H_1=\cn \setminus \{x_1\}$ and $H_2 = \cm \setminus \{y_1\}$. 
	 	The subgraph $H_3$ is defined by taking the bridge $e=x_1y_1$ and the neighboring vertices $\{x_2,x_n,y_2,y_m\}$, i.e. the graph below.	
	 	\begin{center}
	 		\definecolor{ududff}{rgb}{0.30196078431372547,0.30196078431372547,1.}
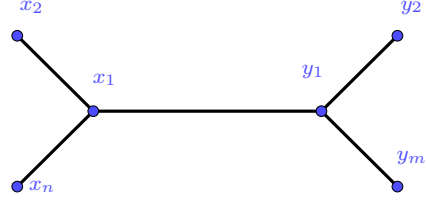
	 		\begin{tikzpicture}[line cap=round,line join=round,>=triangle 45,x=1.0cm,y=1.0cm]
	 		\clip(2.5,1.5) rectangle (9.,4.5);
	 		\draw [line width=1.2pt] (4.,3.)-- (7.,3.);
	 		\draw [line width=1.2pt] (7.,3.)-- (8.,4.);
	 		\draw [line width=1.2pt] (7.,3.)-- (8.,2.);
	 		\draw [line width=1.2pt] (4.,3.)-- (3.,2.);
	 		\draw [line width=1.2pt] (4.,3.)-- (3.,4.);
	 		\begin{scriptsize}
	 		\draw [fill=ududff] (4.,3.) circle (2.0pt);
	 		\draw[color=ududff] (4.15,3.42) node {$x_1$};
	 		\draw [fill=ududff] (7.,3.) circle (2.0pt);
	 		\draw[color=ududff] (6.89,3.52) node {$y_1$};
	 		\draw [fill=ududff] (8.,4.) circle (2.0pt);
	 		\draw[color=ududff] (8.19,4.38) node {$y_2$};
	 		\draw [fill=ududff] (8.,2.) circle (2.0pt);
	 		\draw[color=ududff] (8.19,2.38) node {$y_m$};
	 		\draw [fill=ududff] (3.,2.) circle (2.0pt);
	 		\draw[color=ududff] (3.33,2.) node {$x_n$};
	 		\draw [fill=ududff] (3.,4.) circle (2.0pt);
	 		\draw[color=ududff] (3.19,4.38) node {$x_2$};
	 		\end{scriptsize}
	 		\end{tikzpicture}
	 	\end{center}
	 	Using this decomposition and \autoref{Kalai-Meshulam} we get the inequality 
	 	$$
	 	\reg{R/I(\cncm)} \le \reg{(R/I(H_1))} + \reg{(R/I(H_2))} + \reg{(R/I(H_3))}, 
	 	$$
	 	then have that $H_1$ and $H_2$ are paths of length $n-1$ and $m-1$ respectively, and using \autoref{reg_Forests} we get 
	 	$$
	 	\reg{R/I(\cncm)} \le \Big\lfloor \frac{n}{3} \Big\rfloor + \Big\lfloor \frac{m}{3} \Big\rfloor + 1.  
	 	$$
	 	Finally, in the present case $n \equiv 2 \; (\Mod  3)$ and $m \equiv 2 \; (\Mod 3)$ we have the equality $\nu(\cncm) = \lfloor \frac{n}{3} \rfloor + \lfloor \frac{m}{3} \rfloor +  1$, and the proof follows from \autoref{lower_bound_Katzman}.
	 \end{proof} 	
\end{lem}

\begin{lem}
	\label{reg_case_zero_two} If $n \equiv 0,1 \; (\Mod  3)$ and $m \equiv 2 \;(\Mod  3)$, then
	$$
	\reg{\left(\frac{R}{I(\cncm)}\right)} = \nu(\cncm) + 1 =  \Big\lfloor \frac{n}{3} \Big\rfloor + \Big\lfloor \frac{m}{3} \Big\rfloor + 1.  
	$$
		\begin{proof}
		In this case we will delete the vertex $x_1$ from the cycle $\cn$. We have that $H = (\cncm) \setminus \{x_1\}$ is an induced subgraph of $\cncm$ which is given as the disjoint union of a path of length $n-1$ and a cycle $m$, i.e. $H = P_{n-1} \cup C_m$. 		
		From \autoref{short exact sequence}$(i)$ we get that 
		$$
		\reg{(R/I(\cncm))} \ge \reg{(R/I(H))} = \Big\lfloor \frac{n}{3} \Big\rfloor + \Big\lfloor \frac{m}{3} \Big\rfloor + 1.
		$$
		It follows from \autoref{prop_lower_upper_bound_reg} and the fact that $\lfloor k/3 \rfloor = \lfloor (k-2)/3 \rfloor + 1$ when $k \equiv 0,1 (\Mod 3)$ that 
		$$
		 \reg{R/I(\cncm)}  = \Big\lfloor \frac{n}{3} \Big\rfloor + \Big\lfloor \frac{m}{3} \Big\rfloor + 1.\qedhere
		$$
	\end{proof}
\end{lem}

\begin{theorem}
	\label{formula_reg_dumbbell}
	Let $n,m\ge 3$, then
	$$
	\reg{	\edge{\cncm}}=
	\begin{cases}
	\nu(\cncm)+2 \;  & \text{\normalfont if } n\equiv 0,1 \;(\Mod3), \; m\equiv 2 \;(\Mod3);\\
	\nu(\cncm)+1 \;   & \text{\normalfont otherwise}.\\
	\end{cases}
	$$
	\begin{proof}
		It follows by  \autoref{several_cases_reg}, \autoref{reg_case_two_two} and \autoref{reg_case_zero_two}.
	\end{proof}
\end{theorem}

\subsection{The case $l=3$} \hspace{\fill}\\
Throughout this subsection, we consider the dumbbell graph $ \cnlcm{3} $.
%We will apply  \autoref{TAI_REG_LOWER_BOUND} and \autoref{formula_nu_dumbbell} in our treatment. 

\begin{prop}\label{reg bound of cn p_3 cm}
	Let $n,m\ge 3$, then
	\begin{enumerate}[(i)]
		\item $\reg{I(\cnlcm{3})} \le \nu(\cnlcm{3}) + 2$, \quad if $n, m \equiv 2 \;(\Mod 3)$;
		\item $\reg{I(\cnlcm{3})} = \nu(\cnlcm{3}) + 1$, \quad otherwise.
	\end{enumerate}
	\begin{proof}
		Let $E(P_3)=\{e,e^{\prime}\}$ be the set of the edges of $P_3$, where $e=z_1z_2$ and $e^{\prime}=z_2z_3$ are connected to $\cn$ and $\cm$, respectively. 
		Since $\reg{\left(I(\cn \cup (e^{\prime}\cdot \cm)):e\right)}=\reg{\left(I(P_{n-3} \cup P_{m-1})\right)}$, then \autoref{short exact sequence}$(iii)$ yields the inequality
		$$
		\text{reg}\left(\frac{R}{I(\cnlcm{3})}\right) \le \max\Big\{\reg{\left(  \frac{R}{I(P_{n-3} \cup P_{m-1})} \right)} + 1, \reg{\left(\frac{R}{I(\cn \cup (e^{\prime}\cdot\cm)}\right)} \Big\}.
		$$
		
		From \autoref{nu_path_cycle} and \cite[Lemma 3.2]{REG_UNICYCLIC_GRAPH} follows that $\reg{\left(I(e^{\prime}\cdot\cm)\right)} = \lfloor \frac{m}{3} \rfloor + \lfloor \frac{3 - \xi_3(m)}{3} \rfloor+1$.
		Thus, using \autoref{remark_reg_cycle_formula}, \cite[Lemma 3.2]{HOA_TAM}
		and \autoref{reg_Forests}, we get
		$
		\text{reg}\left(\frac{R}{I(\cnlcm{3})}\right) \le \max \Big\{ \Big\lfloor \frac{n-2}{3} \Big\rfloor + \Big\lfloor \frac{m}{3} \Big\rfloor + 1,  \Big\lfloor \frac{n-2}{3} \Big\rfloor + 1 + \Big\lfloor \frac{m}{3} \Big\rfloor + \Big\lfloor \frac{3-\xi_3(m)}{3} \Big\rfloor\Big\}.
		$
		
		On the other hand, from \autoref{formula_nu_dumbbell} we have that $\nu(\cnlcm{3}) = \lfloor \frac{n}{3} \rfloor + \lfloor \frac{m}{3} \rfloor + \lfloor \frac{4-\xi_3(n)-\xi_3(m)}{3} \rfloor$.
		Therefore, we can check that $\reg{\left(\frac{R}{I(\cnlcm{3})}\right)} \le \nu(\cnlcm{3}) + 1$ when $n, m \equiv 2 \;(\text{mod } 3)$, and that $\reg{\left(\frac{R}{I(\cnlcm{3})}\right)} = \nu(\cnlcm{3})$ in all the remaining cases.
	\end{proof}
\end{prop}

\begin{theorem}\label{reg of cn p3 cm}
	Let $n,m\ge 3$, then
	$$
	\reg{ I(\cnlcm{3}) } = \begin{cases}
	\nu(\cnlcm{3})+2 \qquad \text{\normalfont if } n , m\equiv 2  \;(\Mod 3),\\
	\nu(\cnlcm{3}) +1 \qquad  \text{\normalfont otherwise}.
	\end{cases}
	$$ 
		\begin{proof}
		Using \autoref{reg bound of cn p_3 cm}, then we only need to prove that $\reg{I(\cnlcm{3})} \ge \nu(\cnlcm{3})+2$ in the case $n, m \equiv 2 \;(\text{mod } 3)$. 
		Hence, we assume $n, m \equiv 2 \;(\text{mod } 3)$.
		Let $z_2$ be the middle vertex of $\cnlcm{3}$. By deleting $z_2$ we see that $H=(\cnlcm{3}) \setminus z_2=\cn \cup \cm$ is an induced subgraph of $\cnlcm{3}$. From \autoref{regularity Cn} and \cite[Lemma 3.2]{HOA_TAM}, we have that 
		$$
		\reg{I(H)} = \reg{I(\cn)} + \reg{I(\cm)} - 1 = \nu(\cn) + \nu(\cm) + 3.
		$$
		Since $\nu(\cnlcm{3}) = \nu(C_n) + \nu(C_m) + 1$, then using \autoref{short exact sequence}$(i)$	we get 
		$$
		\reg{I(\cnlcm{3})} \ge \reg{I(H)} = \nu(\cnlcm{3}) + 2.\qedhere
		$$	
	\end{proof}		
\end{theorem} 

\section{Combinatorial characterization of $\reg{(I(G))}$ in terms of $\nu(G)$}\label{reg_bicyclic_graph_section}
Let $G$ be a general bicyclic graph, then its decycling number is smaller or equal than $2$, and so from \autoref{lower_bound_Katzman} and \autoref{Upper_bound_reg_decyc}, we get 
$$
\nu(G) + 1 \le \reg{I(G)} \le \nu(G) + 3.
$$
\begin{exmp}
	\label{examp_reg_nu+3}
	The following graph $G$
	\begin{center}
			\definecolor{ududff}{rgb}{0.30196078431372547,0.30196078431372547,1.}
			\begin{tikzpicture}[line cap=round,line join=round,>=triangle 45,x=1.0cm,y=1.0cm]
			\clip(1.5,1.5) rectangle (10.5,5.5);
			\draw [line width=1.2pt] (2.,4.)-- (2.,2.);
			\draw [line width=1.2pt] (2.,2.)-- (4.,2.);
			\draw [line width=1.2pt] (4.,2.)-- (5.,3.);
			\draw [line width=1.2pt] (5.,3.)-- (4.,4.);
			\draw [line width=1.2pt] (4.,4.)-- (2.,4.);
			\draw [line width=1.2pt] (5.,3.)-- (6.,3.);
			\draw [line width=1.2pt] (6.,3.)-- (7.,3.);
			\draw [line width=1.2pt] (6.,3.)-- (6.,4.);
			\draw [line width=1.2pt] (6.,4.)-- (6.,5.);
			\draw [line width=1.2pt] (7.,3.)-- (8.,4.);
			\draw [line width=1.2pt] (8.,4.)-- (10.,4.);
			\draw [line width=1.2pt] (10.,4.)-- (10.,2.);
			\draw [line width=1.2pt] (10.,2.)-- (8.,2.);
			\draw [line width=1.2pt] (8.,2.)-- (7.,3.);
			\begin{scriptsize}
			\draw [fill=ududff] (2.,4.) circle (2.0pt);
			\draw[color=ududff] (2.19,4.38) node {$x_3$};
			\draw [fill=ududff] (2.,2.) circle (2.0pt);
			\draw[color=ududff] (2.19,2.38) node {$x_4$};
			\draw [fill=ududff] (4.,2.) circle (2.0pt);
			\draw[color=ududff] (3.81,2.46) node {$x_5$};
			\draw [fill=ududff] (5.,3.) circle (2.0pt);
			\draw[color=ududff] (4.47,3.26) node {$x_1$};
			\draw [fill=ududff] (4.,4.) circle (2.0pt);
			\draw[color=ududff] (4.19,4.38) node {$x_2$};
			\draw [fill=ududff] (6.,3.) circle (2.0pt);
			\draw[color=ududff] (5.93,2.74) node {$z_1$};
			\draw [fill=ududff] (7.,3.) circle (2.0pt);
			\draw[color=ududff] (7.55,3.16) node {$y_1$};
			\draw [fill=ududff] (6.,4.) circle (2.0pt);
			\draw[color=ududff] (6.19,4.38) node {$z_2$};
			\draw [fill=ududff] (6.,5.) circle (2.0pt);
			\draw[color=ududff] (6.19,5.38) node {$z_3$};
			\draw [fill=ududff] (8.,4.) circle (2.0pt);
			\draw[color=ududff] (8.15,4.44) node {$y_2$};
			\draw [fill=ududff] (10.,4.) circle (2.0pt);
			\draw[color=ududff] (10.19,4.38) node {$y_3$};
			\draw [fill=ududff] (10.,2.) circle (2.0pt);
			\draw[color=ududff] (10.19,2.38) node {$y_4$};
			\draw [fill=ududff] (8.,2.) circle (2.0pt);
			\draw[color=ududff] (8.19,2.38) node {$y_5$};
			\end{scriptsize}
			\end{tikzpicture}
	\end{center}
	has regularity $\reg{I(G)}=6$ and induced matching number $\nu(G)=3$.
\end{exmp}

In this section, we give a combinatorial characterization of the bicyclic graphs with regularity $\nu(G)+1$, $\nu(G)+2$ and $\nu(G)+3$.
For the rest of the paper, we shall use the term ``dumbbell'' of the bicyclic graph $G$, and it denotes the unique subg
raph of $G$ of the form $\cnplcm$.
The theorem below contains the characterization that we found.

\begin{theorem}
	\label{full_characterization}
	Let $G$ be a bicyclic graph with dumbbell $\cnplcm$.
	The following statements hold.
	\begin{enumerate}[(I)]
		\item Let $ n , m\equiv 0,1  \;(\Mod 3)$, then  $\reg{I(G)} = \nu(G) + 1$.
		\item Let $ n \equiv 0,1  \;(\Mod 3)$ and $m\equiv 2  \;(\Mod 3)$, then  
		$$\nu(G) + 1\le\reg{I(G)} \le \nu(G) + 2,$$ 
		and  $\reg{I(G)} = \nu(G) + 2$ if and only if $\nu(G) = \nu(G\setminus \Gamma_G(C_m))$.
		\item Let $n,m \equiv 2 \; (\Mod3)$ and $l \ge 3$, then  $\nu(G) + 1\le\reg{I(G)} \le \nu(G) + 3$.
		Moreover:
		\begin{enumerate}[(i)]
			\item  $\reg{I(G)} = \nu(G) + 3$   if and only if $\nu\left(G\setminus \Gamma_G(C_n\cup C_m)\right) = \nu(G)$.
			\item  $\reg{I(G)} = \nu(G) + 1$  if and only if the following conditions hold:
			\begin{enumerate}[(a)]
				\item $\nu(G)-\nu(G\setminus \Gamma_G(C_n\cup C_m)) > 1$;
				\item $\nu(G) > \nu(G \setminus \Gamma_G(C_n))$;
				\item $\nu(G) > \nu(G \setminus \Gamma_G(C_m))$.
			\end{enumerate}
		\end{enumerate}
		\item Let $n,m \equiv 2 \; (\Mod3)$ and $l \le 2$, then  $\nu(G) + 1\le\reg{I(G)} \le \nu(G) + 2$.
		If $x$ is an edge on  $P_l$ and  $\LozGraph$ be the Lozin transformation of $G$ with respect to $x$, then  $\reg{I(G)}=\nu(G)+1$  if and only if the following conditions are satisfied: 
		\begin{enumerate}[(a)]
			\item $\nu(\LozGraph)-\nu(\LozGraph\setminus \Gamma_{\LozGraph}(C_n\cup C_m)) > 1$;
			\item $\nu(\LozGraph) > \nu(\LozGraph \setminus \Gamma_{\LozGraph}(C_n))$;
			\item $\nu(\LozGraph) > \nu(\LozGraph \setminus \Gamma_{\LozGraph}(C_m))$.
		\end{enumerate}
	\end{enumerate}
	\begin{proof}
		Statement $\textit{(I)}$ follows from \autoref{general_prop_bicyclic}. 
		In \autoref{thm_only_one_bad_cycle}, $\textit{(II)}$ is proved.
		By \autoref{thm_case_nu+3} and \autoref{charact_thm_nu+1_in_bad_case}, we get $\textit{(III)}$. 
		Finally, from \autoref{cor_case_l<=2}, we obtain $\textit{(IV)}$.
	\end{proof}
\end{theorem}

The following simple remark will be crucial in our treatment.
\begin{rem}\cite[Observation 2.1]{REG_UNICYCLIC_GRAPH}
	\label{leaf_delete_decrease_nu}
	Let $G$ be a graph with a leaf $y$ and its unique neighbor $x$, say $e=\{x,y\}$. 
	If $\{e_1,\ldots,e_s\}$ is an induced matching in $G\setminus N[x]$, then $\{e_1,\ldots,e_s,e \}$ is an induced matching in $G$.
	So we have $\nu(G\setminus N[x]) + 1 \le \nu(G)$.
\end{rem}

\begin{prop}
	\label{general_prop_bicyclic}
	Let $G$ be a bicyclic graph with dumbbell $\cnplcm$. The following statements hold.
	\begin{enumerate}[(i)]
		\item When $ n , m\equiv 0,1  \;(\Mod 3)$, we have $\reg{I(G)} = \nu(G) + 1$.
		\item When $ n \equiv 0,1  \;(\Mod 3)$ and $m\equiv 2  \;(\Mod 3)$, we have $\reg{I(G)} \le \nu(G) + 2$.
		\item When $l \le 2$, we have $\reg{I(G)} \le \nu(G) + 2$.
	\end{enumerate}
	\begin{proof}
		$(i)$ Again, it is enough to prove the upper bound $\reg{I(G)} \le \nu(G) + 1$.
		Let $E^\prime$ be the set of edges $E^\prime=E(G)\setminus E(\cnplcm)$. 
		We proceed by induction on the cardinality of $E^\prime$. 
		If $\lvert E^\prime \rvert = 0$ then the statement follows from \autoref{reg_dumbbell}, so we assume $\lvert E^\prime \rvert > 0$.
		There exists a leaf $y$ in $G$ such that $N[y]=\{x\}$. 
		Let $G^\prime=G \setminus x$ and $G^{\prime\prime}=G\setminus N[x]$, then by \autoref{short exact sequence} we have 
		$$
		\reg{I(G)} \le \max\{\reg{I(G^\prime)}, \reg{I(G^{\prime\prime})}+1 \}.
		$$  
		The graphs $G^\prime$ and $G^{\prime\prime}$ can be either  bicyclic graphs with the same dumbbell $\cnplcm$, or the disjoint union of two unicyclic graphs with cycles $C_n$ and $C_m$, or unicyclic graphs with a cycle $C_r$ ($r=n$ or $r=m$) of the type $r\equiv 0,1  \,(\Mod 3)$, or forests.
		Using either the induction hypothesis,  or \cite[Theorem 1.1]{REG_UNICYCLIC_GRAPH} and \autoref{Kalai-Meshulam}, or \cite[Theorem 1.1]{REG_UNICYCLIC_GRAPH}, or \autoref{reg_Forests}, then we get $\reg{I(G^\prime)} =\nu(G^\prime) + 1$	and $\reg{I(G^{\prime\prime})} =\nu(G^{\prime\prime}) + 1$.	
		Since we have $\nu(G^\prime) \le \nu(G)$ and $\nu(G^{\prime\prime}) + 1 \le \nu(G)$ (by \autoref{leaf_delete_decrease_nu}), then we obtain the required inequality.
		
		$(ii)$ and $(iii)$ follow by the same inductive argument, only changing the fact that $G^\prime$ and $G^{\prime\prime}$ could be unicyclic graphs with cycle $C_r$ of the type $r\equiv 2  \;(\Mod 3)$.	
	\end{proof}
\end{prop}

\begin{rem}
	The inductive process of the previous proposition cannot conclude $\reg{I(G)} \le \nu(G) + 2$
	in the case $l\ge 3$.
	Here we may encounter two disjoint subgraphs $G_1$ and $G_2$ with $\reg{I(G_i)} = \nu(G_i)+2$, which implies $\reg{I(G_1 \cup G_2)} = \nu(G_1\cup G_2) + 3$.
	This is exactly the case of \autoref{examp_reg_nu+3}.
	
	An alternative proof of the inequality $\reg{I(G)} \le \nu(G) + 3$ can be given by using the same inductive technique of \autoref{general_prop_bicyclic}.
\end{rem}

For the rest of the paper we shall use the following notation.
\begin{nota}
	Let $G$ be a graph, $H \subset G$ be a subgraph, and $v$ and $w$ be vertices of $G$. 
	Then, we assume the following:
	\begin{enumerate}[(i)]
	%	\item $\deg(v)$ denotes the degree of $v$ (i.e. the number of edges incident to $v$).
		\item  $d(v,w)$ denotes the length (i.e., the number of edges) of a minimal path between $v$ and $w$.
		In particular, $d(v,v)=0$.
		\item $d(v,H)$ denotes the minimal distance from the vertex $v$ to the subgraph $H$, that is 
		$$
		d(v, H)=\min\{ d(v, w) \mid w \in H \}.
		$$
		In particular, $d(v,H)=0$ if and only if $v \in H$.
		\item Let $H^{\prime} \subset G$ be a subgraph, then the distance between $H$ and $H^{\prime}$ is given by
		$$
		d(H,H^{\prime})=\min\{d(v,H^{\prime})\mid v \in H \}.
		$$
		In particular, $d(H,H^{\prime})=0$ if and only if $H \cap H^{\prime}\neq\emptyset$.
		\item $\Gamma_G(H)$ denotes the subset of vertices
		$$
		\Gamma_G(H) = \{v \in G \mid d(v, H) = 1 \}.
		$$ 	
		\item 	In the case $k > 0$, $S_{G,k}(H)$ denotes the induced subgraph given by restricting to the vertex set 
		$$
		V(S_{G,k}(H)) = \{v \in G\mid  d(v, H) \ge k \}.
		$$
		\item $S_{G,0}$ denotes the subgraph given by the vertex set 
		$$
		V(S_{G,0}(H)) = \{v \in G\mid  d(v, H) > 0 \text{ or } \deg(v) \ge 3 \}.
		$$
		and the edge set
		\begin{align*}
			E(S_{G,0}(H)) = \{ (v,w) \in E(G&) \mid v,w \in V(S_{G,0}(H)) \} \\ &\setminus \{ (v,w) \in E(G) \mid v,w \in H \}.
		\end{align*}
	\end{enumerate}
\end{nota}

We clarify the previous notation in the following example.

\begin{exmp}
	\begin{enumerate}[(i)]
		\item 	Let $G$ be the graph of \autoref{examp_reg_nu+3} and $H=C_5\cup C_5$ be the subgraph given by the two cycles of length $5$. 
		Then, we have that $\Gamma_G(H)$ is the set containing the vertex in the middle of the bridge joining the two cycles, that $S_{G,0}(H)$ is a graph of the form
		\vspace*{-.1cm}
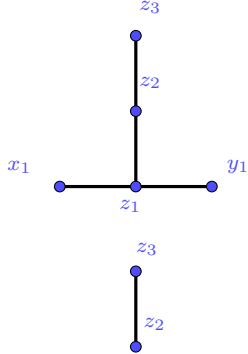
		\begin{center}
			\definecolor{ududff}{rgb}{0.30196078431372547,0.30196078431372547,1.}
			\begin{tikzpicture}[line cap=round,line join=round,>=triangle 45,x=1.0cm,y=1.0cm]
			\clip(4.,2.2) rectangle (8.,5.5);
			\draw [line width=1.2pt] (5.,3.)-- (6.,3.);
			\draw [line width=1.2pt] (6.,3.)-- (7.,3.);
			\draw [line width=1.2pt] (6.,3.)-- (6.,4.);
			\draw [line width=1.2pt] (6.,4.)-- (6.,5.);
			\begin{scriptsize}
			\draw [fill=ududff] (5.,3.) circle (2.0pt);
			\draw[color=ududff] (4.47,3.26) node {$x_1$};
			\draw [fill=ududff] (6.,3.) circle (2.0pt);
			\draw[color=ududff] (5.93,2.74) node {$z_1$};
			\draw [fill=ududff] (7.,3.) circle (2.0pt);
			\draw[color=ududff] (7.35,3.26) node {$y_1$};
			\draw [fill=ududff] (6.,4.) circle (2.0pt);
			\draw[color=ududff] (6.19,4.38) node {$z_2$};
			\draw [fill=ududff] (6.,5.) circle (2.0pt);
			\draw[color=ududff] (6.19,5.38) node {$z_3$};
			\end{scriptsize}
			\end{tikzpicture}
		\end{center}
		\vspace*{-.2cm}
		and that the graph
		\vspace*{-.7cm}
		\begin{center}
			\definecolor{ududff}{rgb}{0.30196078431372547,0.30196078431372547,1.}
			\begin{tikzpicture}[line cap=round,line join=round,>=triangle 45,x=1.0cm,y=1.0cm]
			\clip(5.,3.6) rectangle (7.,5.5);
			\draw [line width=1.2pt] (6.,4.)-- (6.,5.);
			\begin{scriptsize}
			\draw [fill=ududff] (6.,4.) circle (2.0pt);
			\draw[color=ududff] (6.2492337881991075,4.29805366134922) node {$z_2$};
			\draw [fill=ududff] (6.,5.) circle (2.0pt);
			\draw[color=ududff] (6.1492337881991075,5.297477070920621) node {$z_3$};
			\end{scriptsize}
			\end{tikzpicture}
		\end{center}
		\vspace*{-.2cm}
		represents $S_{G,2}(H)$.
		
		\item Let $G$ be the graph given by
		\begin{center}
			\definecolor{ududff}{rgb}{0.30196078431372547,0.30196078431372547,1.}
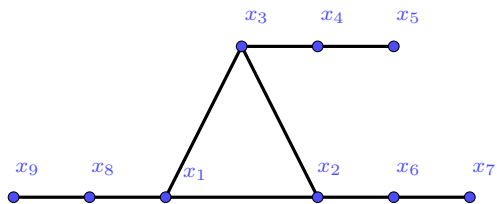
			\begin{tikzpicture}[line cap=round,line join=round,>=triangle 45,x=1.0cm,y=1.0cm]
			\clip(0.5,2.5) rectangle (8.,5.5);
			\draw [line width=1.2pt] (3.,3.)-- (4.,5.);
			\draw [line width=1.2pt] (4.,5.)-- (5.,3.);
			\draw [line width=1.2pt] (5.,3.)-- (3.,3.);
			\draw [line width=1.2pt] (4.,5.)-- (5.,5.);
			\draw [line width=1.2pt] (5.,5.)-- (6.,5.);
			\draw [line width=1.2pt] (5.,3.)-- (6.,3.);
			\draw [line width=1.2pt] (6.,3.)-- (7.,3.);
			\draw [line width=1.2pt] (3.,3.)-- (2.,3.);
			\draw [line width=1.2pt] (2.,3.)-- (1.,3.);
			\begin{scriptsize}
			\draw [fill=ududff] (3.,3.) circle (2.0pt);
			\draw[color=ududff] (3.3895235394214446,3.312864435621105) node {$x_1$};
			\draw [fill=ududff] (4.,5.) circle (2.0pt);
			\draw[color=ududff] (4.19802609188883,5.389325014180379) node {$x_3$};
			\draw [fill=ududff] (5.,3.) circle (2.0pt);
			\draw[color=ududff] (5.156392512762341,3.3727623369257) node {$x_2$};
			\draw [fill=ududff] (5.,5.) circle (2.0pt);
			\draw[color=ududff] (5.196324446965404,5.389325014180379) node {$x_4$};
			\draw [fill=ududff] (6.,5.) circle (2.0pt);
			\draw[color=ududff] (6.1946228020419785,5.389325014180379) node {$x_5$};
			\draw [fill=ududff] (6.,3.) circle (2.0pt);
			\draw[color=ududff] (6.1946228020419785,3.3727623369257) node {$x_6$};
			\draw [fill=ududff] (7.,3.) circle (2.0pt);
			\draw[color=ududff] (7.192921157118552,3.3727623369257) node {$x_7$};
			\draw [fill=ududff] (2.,3.) circle (2.0pt);
			\draw[color=ududff] (2.1814634146341505,3.3727623369257) node {$x_8$};
			\draw [fill=ududff] (1.,3.) circle (2.0pt);
			\draw[color=ududff] (1.1831650595575767,3.3727623369257) node {$x_9$};
			\end{scriptsize}
			\end{tikzpicture}
		\end{center}
	and $H$ be the triangle induced by the vertices $\{x_1,x_2,x_3\}$.
	Then, we have that $\Gamma_G(H)=\{x_4, x_6, x_8\}$, that $S_{G,0}(H)$ is a graph of the form
	\begin{center}
		\definecolor{ududff}{rgb}{0.30196078431372547,0.30196078431372547,1.}
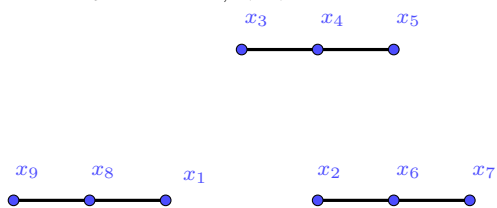
		\begin{tikzpicture}[line cap=round,line join=round,>=triangle 45,x=1.0cm,y=1.0cm]
		\clip(0.5,2.5) rectangle (8.,5.5);
		%\draw [line width=1.2pt] (3.,3.)-- (4.,5.);
		%\draw [line width=1.2pt] (4.,5.)-- (5.,3.);
		%\draw [line width=1.2pt] (5.,3.)-- (3.,3.);
		\draw [line width=1.2pt] (4.,5.)-- (5.,5.);
		\draw [line width=1.2pt] (5.,5.)-- (6.,5.);
		\draw [line width=1.2pt] (5.,3.)-- (6.,3.);
		\draw [line width=1.2pt] (6.,3.)-- (7.,3.);
		\draw [line width=1.2pt] (3.,3.)-- (2.,3.);
		\draw [line width=1.2pt] (2.,3.)-- (1.,3.);
		\begin{scriptsize}
		\draw [fill=ududff] (3.,3.) circle (2.0pt);
		\draw[color=ududff] (3.3895235394214446,3.312864435621105) node {$x_1$};
		\draw [fill=ududff] (4.,5.) circle (2.0pt);
		\draw[color=ududff] (4.19802609188883,5.389325014180379) node {$x_3$};
		\draw [fill=ududff] (5.,3.) circle (2.0pt);
		\draw[color=ududff] (5.156392512762341,3.3727623369257) node {$x_2$};
		\draw [fill=ududff] (5.,5.) circle (2.0pt);
		\draw[color=ududff] (5.196324446965404,5.389325014180379) node {$x_4$};
		\draw [fill=ududff] (6.,5.) circle (2.0pt);
		\draw[color=ududff] (6.1946228020419785,5.389325014180379) node {$x_5$};
		\draw [fill=ududff] (6.,3.) circle (2.0pt);
		\draw[color=ududff] (6.1946228020419785,3.3727623369257) node {$x_6$};
		\draw [fill=ududff] (7.,3.) circle (2.0pt);
		\draw[color=ududff] (7.192921157118552,3.3727623369257) node {$x_7$};
		\draw [fill=ududff] (2.,3.) circle (2.0pt);
		\draw[color=ududff] (2.1814634146341505,3.3727623369257) node {$x_8$};
		\draw [fill=ududff] (1.,3.) circle (2.0pt);
		\draw[color=ududff] (1.1831650595575767,3.3727623369257) node {$x_9$};
		\end{scriptsize}
		\end{tikzpicture}
	\end{center}
	and that the graph
	\begin{center}
		\definecolor{ududff}{rgb}{0.30196078431372547,0.30196078431372547,1.}
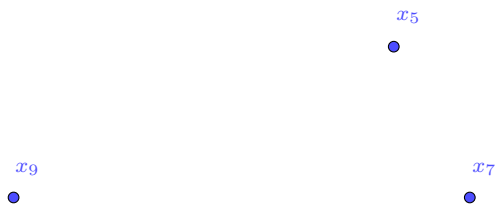
		\begin{tikzpicture}[line cap=round,line join=round,>=triangle 45,x=1.0cm,y=1.0cm]
		\clip(0.5,2.5) rectangle (8.,5.5);
		%\draw [line width=1.2pt] (3.,3.)-- (4.,5.);
		%\draw [line width=1.2pt] (4.,5.)-- (5.,3.);
		%\draw [line width=1.2pt] (5.,3.)-- (3.,3.);
		%\draw [line width=1.2pt] (4.,5.)-- (5.,5.);
		%\draw [line width=1.2pt] (5.,5.)-- (6.,5.);
		%\draw [line width=1.2pt] (5.,3.)-- (6.,3.);
		%\draw [line width=1.2pt] (6.,3.)-- (7.,3.);
		%\draw [line width=1.2pt] (3.,3.)-- (2.,3.);
		%\draw [line width=1.2pt] (2.,3.)-- (1.,3.);
		\begin{scriptsize}
		%\draw [fill=ududff] (3.,3.) circle (2.0pt);
		%\draw[color=ududff] (3.3895235394214446,3.312864435621105) node {$x_1$};
		%\draw [fill=ududff] (4.,5.) circle (2.0pt);
		%\draw[color=ududff] (4.19802609188883,5.389325014180379) node {$x_3$};
		%\draw [fill=ududff] (5.,3.) circle (2.0pt);
		%\draw[color=ududff] (5.156392512762341,3.3727623369257) node {$x_2$};
		%\draw [fill=ududff] (5.,5.) circle (2.0pt);
		%\draw[color=ududff] (5.196324446965404,5.389325014180379) node {$x_4$};
		\draw [fill=ududff] (6.,5.) circle (2.0pt);
		\draw[color=ududff] (6.1946228020419785,5.389325014180379) node {$x_5$};
		%\draw [fill=ududff] (6.,3.) circle (2.0pt);
		%\draw[color=ududff] (6.1946228020419785,3.3727623369257) node {$x_6$};
		\draw [fill=ududff] (7.,3.) circle (2.0pt);
		\draw[color=ududff] (7.192921157118552,3.3727623369257) node {$x_7$};
		%\draw [fill=ududff] (2.,3.) circle (2.0pt);
		%\draw[color=ududff] (2.1814634146341505,3.3727623369257) node {$x_8$};
		\draw [fill=ududff] (1.,3.) circle (2.0pt);
		\draw[color=ududff] (1.1831650595575767,3.3727623369257) node {$x_9$};
		\end{scriptsize}
		\end{tikzpicture}
	\end{center}
	represents $S_{G,2}(H)$.
	\end{enumerate}
\end{exmp}

We have already computed $\reg{I(G)}$ in the case $n,m\equiv 0,1 \; (\text{mod}\;3)$, for the remaining cases we shall divide this section into subsections.

\subsection{Case I}\hspace{\fill}\\
In this subsection we shall focus on the case $n \equiv 0,1 \; (\text{mod}\;3)$
and $m \equiv 2 \; (\text{mod}\;3)$.
This case turns out to be almost identical to a unicyclic graph, and our treatment is influenced by \cite[Section 3]{REG_UNICYCLIC_GRAPH}. 

\begin{nota}
	\label{only_one_bad_cycle}
	Let $G$ be a bicyclic graph with dumbbell $\cnplcm$ such that $n \equiv 0,1 \; (\Mod 3)$
	and $m \equiv 2 \; (\Mod 3)$.
	We shall denote by $F_1,\ldots,F_c$ the connected components of $S_{G,0}(C_m)$, and in this case each $F_i$ is either a tree or a unicyclic graph with cycle $C_n$ (and $n \equiv 0,1 \; (\Mod 3)$).
	Then, the graph $S_{G,2}(C_m)$ can be given as the union of the components $H_1,\ldots,H_c$, where each one is defined as 
	$$
	H_i = F_i \setminus \{v \in G \mid d(v, C_m) \le 1 \}.
	$$		
	We note that each $H_i$ can be a non-connected graph or even the empty graph.
\end{nota}

\begin{rem}	
	\label{rem_Case_I}
	The following statements hold.
	\begin{enumerate}[(i)]
		\item The graph $G\setminus \Gamma_G(C_m)$ has a decomposition of the form
		$$
		G\setminus \Gamma_G(C_m) = C_m \bigcup \left(\bigcup_{i=1}^cH_i\right),
		$$
		and in particular
		$$
		\nu\left(G\setminus \Gamma_G(C_m)\right) = \nu\left(C_m\right) + \sum_{i=1}^{c}\nu\left(H_i\right)
		$$
		because $d(C_m,H_i)\ge 2$ for all $1\le i \le  c$ and $d(H_i,H_j)\ge 2$ for all $1\le i<j \le c$. 
		\item For each $i=1,\ldots,c$, we have that $\lvert F_i \cap C_m \rvert=1$.
	\end{enumerate}

\end{rem}

\begin{exmp}
	Let $G$ be the graph 
	\begin{center}
	 	\definecolor{ududff}{rgb}{0.30196078431372547,0.30196078431372547,1.}
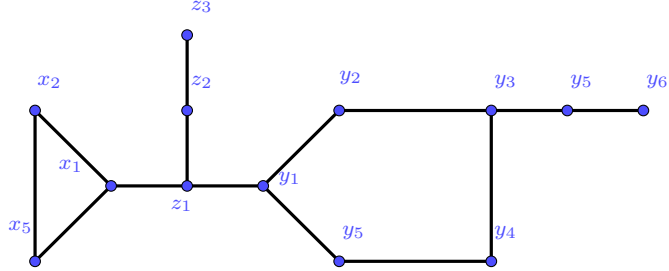
	 	\begin{tikzpicture}[line cap=round,line join=round,>=triangle 45,x=1.0cm,y=1.0cm]
	 	\clip(3.5,1.5) rectangle (13.,5.5);
	 	\draw [line width=1.2pt] (5.,3.)-- (6.,3.);
	 	\draw [line width=1.2pt] (6.,3.)-- (7.,3.);
	 	\draw [line width=1.2pt] (6.,3.)-- (6.,4.);
	 	\draw [line width=1.2pt] (6.,4.)-- (6.,5.);
	 	\draw [line width=1.2pt] (7.,3.)-- (8.,4.);
	 	\draw [line width=1.2pt] (8.,4.)-- (10.,4.);
	 	\draw [line width=1.2pt] (10.,4.)-- (10.,2.);
	 	\draw [line width=1.2pt] (10.,2.)-- (8.,2.);
	 	\draw [line width=1.2pt] (8.,2.)-- (7.,3.);
	 	\draw [line width=1.2pt] (4.,2.)-- (5.,3.);
	 	\draw [line width=1.2pt] (4.,4.)-- (5.,3.);
	 	\draw [line width=1.2pt] (4.,4.)-- (4.,2.);
	 	\draw [line width=1.2pt] (10.,4.)-- (11.,4.);
	 	\draw [line width=1.2pt] (11.,4.)-- (12.,4.);
	 	\begin{scriptsize}
	 	\draw [fill=ududff] (4.,2.) circle (2.0pt);
	 	\draw[color=ududff] (3.81,2.46) node {$x_5$};
	 	\draw [fill=ududff] (5.,3.) circle (2.0pt);
	 	\draw[color=ududff] (4.47,3.26) node {$x_1$};
	 	\draw [fill=ududff] (4.,4.) circle (2.0pt);
	 	\draw[color=ududff] (4.19,4.38) node {$x_2$};
	 	\draw [fill=ududff] (6.,3.) circle (2.0pt);
	 	\draw[color=ududff] (5.93,2.74) node {$z_1$};
	 	\draw [fill=ududff] (7.,3.) circle (2.0pt);
	 	\draw[color=ududff] (7.35,3.1) node {$y_1$};
	 	\draw [fill=ududff] (6.,4.) circle (2.0pt);
	 	\draw[color=ududff] (6.19,4.38) node {$z_2$};
	 	\draw [fill=ududff] (6.,5.) circle (2.0pt);
	 	\draw[color=ududff] (6.19,5.38) node {$z_3$};
	 	\draw [fill=ududff] (8.,4.) circle (2.0pt);
	 	\draw[color=ududff] (8.15,4.44) node {$y_2$};
	 	\draw [fill=ududff] (10.,4.) circle (2.0pt);
	 	\draw[color=ududff] (10.19,4.38) node {$y_3$};
	 	\draw [fill=ududff] (10.,2.) circle (2.0pt);
	 	\draw[color=ududff] (10.19,2.38) node {$y_4$};
	 	\draw [fill=ududff] (8.,2.) circle (2.0pt);
	 	\draw[color=ududff] (8.19,2.38) node {$y_5$};
	 	\draw [fill=ududff] (11.,4.) circle (2.0pt);
	 	\draw[color=ududff] (11.19,4.38) node {$y_5$};
	 	\draw [fill=ududff] (12.,4.) circle (2.0pt);
	 	\draw[color=ududff] (12.19,4.38) node {$y_6$};
	 	\end{scriptsize}
	 	\end{tikzpicture}			
	\end{center}
	and $C_5$ be the cycle given by $\{y_1,y_2,y_3,y_4,y_5\}$.
 	We have that $\Gamma_G(C_5)=\{z_1,y_5\}$.
 	The graph $S_{G,0}(C_5)$  is given by 
	\begin{center}
		\definecolor{ududff}{rgb}{0.30196078431372547,0.30196078431372547,1.}
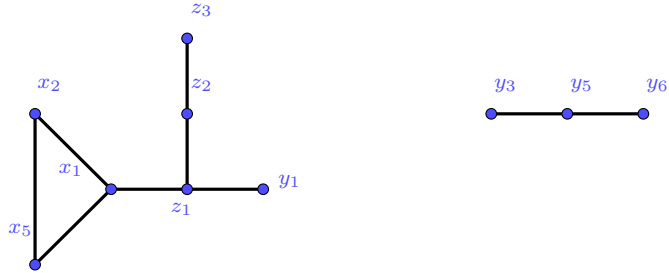
		\begin{tikzpicture}[line cap=round,line join=round,>=triangle 45,x=1.0cm,y=1.0cm]
		\clip(3.5,1.5) rectangle (13.,5.5);
		\draw [line width=1.2pt] (5.,3.)-- (6.,3.);
		\draw [line width=1.2pt] (6.,3.)-- (7.,3.);
		\draw [line width=1.2pt] (6.,3.)-- (6.,4.);
		\draw [line width=1.2pt] (6.,4.)-- (6.,5.);
		%\draw [line width=1.2pt] (7.,3.)-- (8.,4.);
		%\draw [line width=1.2pt] (8.,4.)-- (10.,4.);
		%\draw [line width=1.2pt] (10.,4.)-- (10.,2.);
		%\draw [line width=1.2pt] (10.,2.)-- (8.,2.);
		%\draw [line width=1.2pt] (8.,2.)-- (7.,3.);
		\draw [line width=1.2pt] (4.,2.)-- (5.,3.);
		\draw [line width=1.2pt] (4.,4.)-- (5.,3.);
		\draw [line width=1.2pt] (4.,4.)-- (4.,2.);
		\draw [line width=1.2pt] (10.,4.)-- (11.,4.);
		\draw [line width=1.2pt] (11.,4.)-- (12.,4.);
		\begin{scriptsize}
		\draw [fill=ududff] (4.,2.) circle (2.0pt);
		\draw[color=ududff] (3.81,2.46) node {$x_5$};
		\draw [fill=ududff] (5.,3.) circle (2.0pt);
		\draw[color=ududff] (4.47,3.26) node {$x_1$};
		\draw [fill=ududff] (4.,4.) circle (2.0pt);
		\draw[color=ududff] (4.19,4.38) node {$x_2$};
		\draw [fill=ududff] (6.,3.) circle (2.0pt);
		\draw[color=ududff] (5.93,2.74) node {$z_1$};
		\draw [fill=ududff] (7.,3.) circle (2.0pt);
		\draw[color=ududff] (7.35,3.1) node {$y_1$};
		\draw [fill=ududff] (6.,4.) circle (2.0pt);
		\draw[color=ududff] (6.19,4.38) node {$z_2$};
		\draw [fill=ududff] (6.,5.) circle (2.0pt);
		\draw[color=ududff] (6.19,5.38) node {$z_3$};
	%	\draw [fill=ududff] (8.,4.) circle (2.0pt);
	%	\draw[color=ududff] (8.15,4.44) node {$y_2$};
		\draw [fill=ududff] (10.,4.) circle (2.0pt);
		\draw[color=ududff] (10.19,4.38) node {$y_3$};
		%\draw [fill=ududff] (10.,2.) circle (2.0pt);
		%\draw[color=ududff] (10.19,2.38) node {$y_4$};
		%\draw [fill=ududff] (8.,2.) circle (2.0pt);
		%\draw[color=ududff] (8.19,2.38) node {$y_5$};
		\draw [fill=ududff] (11.,4.) circle (2.0pt);
		\draw[color=ududff] (11.19,4.38) node {$y_5$};
		\draw [fill=ududff] (12.,4.) circle (2.0pt);
		\draw[color=ududff] (12.19,4.38) node {$y_6$};
		\end{scriptsize}
		\end{tikzpicture}			
	\end{center}
	with connected components $F_1=\{y_1,z_1,z_2,z_3,x_1,x_2,x_5\}$ and $F_2=\{y_3,y_4,y_5\}$.
	The graph $S_{G,2}(C_5)$ is given by
		\begin{center}
		\definecolor{ududff}{rgb}{0.30196078431372547,0.30196078431372547,1.}
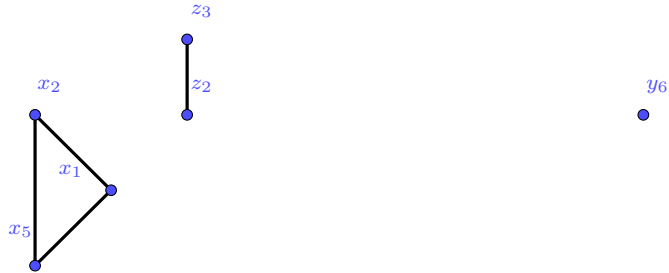
		\begin{tikzpicture}[line cap=round,line join=round,>=triangle 45,x=1.0cm,y=1.0cm]
		\clip(3.5,1.5) rectangle (13.,5.5);
		%\draw [line width=1.2pt] (5.,3.)-- (6.,3.);
		%\draw [line width=1.2pt] (6.,3.)-- (7.,3.);
		%\draw [line width=1.2pt] (6.,3.)-- (6.,4.);
		\draw [line width=1.2pt] (6.,4.)-- (6.,5.);
		%\draw [line width=1.2pt] (7.,3.)-- (8.,4.);
		%\draw [line width=1.2pt] (8.,4.)-- (10.,4.);
		%\draw [line width=1.2pt] (10.,4.)-- (10.,2.);
		%\draw [line width=1.2pt] (10.,2.)-- (8.,2.);
		%\draw [line width=1.2pt] (8.,2.)-- (7.,3.);
		\draw [line width=1.2pt] (4.,2.)-- (5.,3.);
		\draw [line width=1.2pt] (4.,4.)-- (5.,3.);
		\draw [line width=1.2pt] (4.,4.)-- (4.,2.);
		%\draw [line width=1.2pt] (10.,4.)-- (11.,4.);
		%\draw [line width=1.2pt] (11.,4.)-- (12.,4.);
		\begin{scriptsize}
		\draw [fill=ududff] (4.,2.) circle (2.0pt);
		\draw[color=ududff] (3.81,2.46) node {$x_5$};
		\draw [fill=ududff] (5.,3.) circle (2.0pt);
		\draw[color=ududff] (4.47,3.26) node {$x_1$};
		\draw [fill=ududff] (4.,4.) circle (2.0pt);
		\draw[color=ududff] (4.19,4.38) node {$x_2$};
		%\draw [fill=ududff] (6.,3.) circle (2.0pt);
		%\draw[color=ududff] (5.93,2.74) node {$z_1$};
		%\draw [fill=ududff] (7.,3.) circle (2.0pt);
		%\draw[color=ududff] (7.35,3.1) node {$y_1$};
		\draw [fill=ududff] (6.,4.) circle (2.0pt);
		\draw[color=ududff] (6.19,4.38) node {$z_2$};
		\draw [fill=ududff] (6.,5.) circle (2.0pt);
		\draw[color=ududff] (6.19,5.38) node {$z_3$};
		%	\draw [fill=ududff] (8.,4.) circle (2.0pt);
		%	\draw[color=ududff] (8.15,4.44) node {$y_2$};
		%\draw [fill=ududff] (10.,4.) circle (2.0pt);
		%\draw[color=ududff] (10.19,4.38) node {$y_3$};
		%\draw [fill=ududff] (10.,2.) circle (2.0pt);
		%\draw[color=ududff] (10.19,2.38) node {$y_4$};
		%\draw [fill=ududff] (8.,2.) circle (2.0pt);
		%\draw[color=ududff] (8.19,2.38) node {$y_5$};
		%\draw [fill=ududff] (11.,4.) circle (2.0pt);
		%\draw[color=ududff] (11.19,4.38) node {$y_5$};
		\draw [fill=ududff] (12.,4.) circle (2.0pt);
		\draw[color=ududff] (12.19,4.38) node {$y_6$};
		\end{scriptsize}
		\end{tikzpicture}			
	\end{center}
	and following our notations we have $H_1=\{x_1,x_2,x_5,z_2,z_3\}$ and $H_2=\{y_6\}$.
\end{exmp}

\begin{lem}
	\label{components_case_only_one_bad}
	Adopt \autoref{only_one_bad_cycle}.
	If $\nu(H_i)=\nu(F_i)$ for all $1 \le i \le c$, then $\nu(G \setminus \Gamma_G(C_m))=\nu(G)$.
	\begin{proof}
		Follows identically to \cite[Lemma 3.5]{REG_UNICYCLIC_GRAPH}.
	\end{proof}
\end{lem}

\begin{prop}
	\label{implic_one_bad_cycle}
	Adopt \autoref{only_one_bad_cycle}.
	If $\nu(G \setminus \Gamma_G(C_m)) < \nu(G)$ then $\reg{I(G)} = \nu(G) + 1$.
	\begin{proof}
		Once more, we shall only prove that $\reg{I(G)} \le \nu(G) + 1$.
		Assume that $\nu(G \setminus \Gamma_G(C_m)) < \nu(G)$, then the contrapositive of \autoref{components_case_only_one_bad} implies that there exists some $i$ with $\nu(H_i) < \nu(F_i)$.
		
		Fix $i$  such that $\nu(H_i) < \nu(F_i)$.
		From \autoref{rem_Case_I}$(ii)$, let $x$ be the vertex in $F_i \cap C_m$. 
		Let us use the notations $G^\prime=G\setminus x$ and $G^{\prime\prime} =G \setminus N[x]$.
		Again, we have the inequality 
		$$
		\reg{I(G)} \le \max\{\reg{I(G^\prime)}, \reg{I(G^{\prime\prime})}+1 \}.
		$$  
		Note that both $G^\prime$ and $G^{\prime\prime}$ can be  either unicyclic graphs with cycle $C_n$ (and $n \equiv 0,1 \; (\text{mod}\;3)$), or forests.
		Hence, from \cite[Theorem 1.1]{REG_UNICYCLIC_GRAPH} and \autoref{reg_Forests} we get that $\reg{I(G^\prime)}=\nu(G^\prime) + 1$ and $\reg{I(G^{\prime\prime})}=\nu(G^{\prime\prime})+1.$

		In the case of $G^\prime$, we have that $\reg{I(G^\prime)} = \nu(G^\prime) + 1 \le \nu(G) + 1$.
		Let $H$ be the induced subgraph of $G$ obtained by deleting the vertices of $F_i \cup N_G[x]$.
		Then we have $G^{\prime\prime}=H \cup H_i$. 
		Let $\matching_1$ and $\matching_2$ be maximal induced matchings in $H$ and $H_i$, respectively, then $\nu(G^{\prime\prime}) = \lvert \matching_1\rvert + \lvert \matching_2\rvert$ because $d(H,H_i)\ge2$.  
		By the condition $\nu(F_i) > \nu(H_i)$ then there exists a maximal induced matching $\matching_3$ in $F_i$, such that $\lvert \matching_3 \rvert > \lvert \matching_2 \rvert$.
		From the fact that $H \cup F_i$ is an induced subgraph in $G$ and $H \cap F_i = \emptyset$, then we get 
		$$
		\nu(G) \ge \nu(H\cup F_i)=\lvert \matching_1\rvert + \lvert \matching_3\rvert > \lvert \matching_1\rvert + \lvert \matching_2\rvert=\nu(G^{\prime\prime}).
		$$
		Hence $\reg{I(G^{\prime\prime})}=\nu(G^{\prime\prime})+1 \le \nu(G)$, and so we get the statement of the proposition.
	\end{proof}
\end{prop}

\begin{theorem}
	\label{thm_only_one_bad_cycle}
	Let $G$ be a bicyclic graph with dumbbell $\cnplcm$ such that $n \equiv 0,1 \; (\Mod 3)$
	and $m \equiv 2 \; (\Mod 3)$.
	Then the following statements hold.
	\begin{enumerate}[(i)]
		\item $\nu(G) + 1 \le \reg{I(G)} \le \nu(G) + 2$;
		\item $\reg{I(G)} = \nu(G) + 2$ if and only if $\nu(G) = \nu(G\setminus \Gamma_G(C_m))$.
	\end{enumerate}
	\begin{proof}
		In \autoref{general_prop_bicyclic} we proved $(i)$.	
		In order to prove $(ii)$, we only need to show that $\nu(G\setminus \Gamma_G(C_m)) = \nu(G)$ implies $\reg{I(G)}\ge\nu(G) +2$, because the inverse implication follows from \autoref{implic_one_bad_cycle}.
		
		From \autoref{rem_Case_I}$(i)$, $G\setminus \Gamma_G(C_m) = C_m \cup (\cup_{i=1}^cH_i$) where each $H_i$ is either a forest or a unicyclic graph with cycle $C_n$ (and $n \equiv 0,1 \; (\text{mod}\;3)$).
		Then, from \cite[Theorem 1.1]{REG_UNICYCLIC_GRAPH} and \autoref{reg_Forests} we get 
		\begin{align*}
		\reg{I(G\setminus \Gamma_G(C_m))} &= \reg{I(C_m)} + \reg{I(\cup_{i=1}^cH_i)} - 1\\
		&= (\nu(C_m) + 2) +(\nu(\cup_{i=1}^cH_i) + 1) -1 \\
		&=\nu(G \setminus \Gamma_G(C_m)) + 2\\
		&=\nu(G) + 2.
		\end{align*}
		Finally, since $G\setminus \Gamma_G(C_m)$ is an induced subgraph of $G$ then we have $\reg{I(G)} \ge \nu(G) + 2$.
	\end{proof}
\end{theorem}

\subsection{Case II}\hspace{\fill}\\
The object of study of this subsection is the case where $n,m \equiv 2 \; (\text{mod}\;3)$, $l \ge 3$, and in particular when $\reg{I(G)}=\nu(G)+3$.
More specifically, we shall give necessary and sufficient conditions for the equality $\reg{I(G)}=\nu(G)+3$.

\begin{nota}
	\label{notation_trees_in_bicyclic}
	Let $G$ be a bicyclic graph with dumbbell $\cnplcm$ such that $n,m \equiv 2 \; (\Mod 3)$ and $l \ge 3$. 
	As in \autoref{only_one_bad_cycle}, let $F_1,\ldots,F_c$ be the components of the graph $S_{G,0}(C_n)$.
	We order the $F_i$'s in such a way that $F_1$ is a unicyclic 
	graph with cycle $C_m$, and for all $i > 1$ we have that $F_i$ is a tree.
	The graph $S_{G,2}(C_n)$  can be decomposed in components  $H_1, \ldots, H_c$ where 
	$$
	H_i = F_i \setminus \{v \in G \mid d(v, C_n) \le 1 \}.
	$$ 
	%Now we make the same process to the unicyclic graph $H_1$ with cycle $C_m$.
	%Let $P_1,\ldots,P_d$ be the components of the graph $S_{H_1,0}(C_m)$.
	%We assume that $P_1$ corresponds to the initial bridge connecting $C_n$ and $C_m$, and we remark that we are using an abuse of notation since it might be the empty graph.
	%We decompose the graph $S_{H_1,2}(C_m)$ into components $Q_1,\ldots Q_d$ given by 
	%$$
	%Q_i = P_i \setminus \{v \in H_1 \mid d(v, C_m) \le 1 \}.
	%$$
\end{nota}

\begin{rem}
	\label{rem_Case_II}
	From the previous notation get the following simple remarks.
	\begin{enumerate}[(i)]
		\item The graph $G \setminus \Gamma_{G}(C_n)$ has a decomposition of the form 
		$$
		G \setminus \Gamma_{G}(C_n)= C_n \bigcup \left(\bigcup_{i=1}^cH_i\right),
		$$
		and in particular
		$$
		\nu\left(G\setminus \Gamma_G(C_n)\right) = \nu\left(C_n\right) + \sum_{i=1}^{c}\nu\left(H_i\right)
		$$
		because $d(C_n,H_i)\ge 2$ for all $1\le i \le  c$ and $d(H_i,H_j)\ge 2$ for all $1\le i<j \le c$. 
		\item Similarly, the graph $G \setminus \Gamma_G(C_n\cup C_m)$ has a decomposition of the form
		$$
		G \setminus \Gamma_G(C_n\cup C_m)=C_n \bigcup \left(\bigcup_{i=2}^c H_i \right)  \bigcup \left( H_1 \setminus \Gamma_{H_1}(C_m) \right),
		$$
		and in particular
		$$
		\nu(G \setminus \Gamma_G(C_n\cup C_m)) = \nu(C_n) +  \sum_{i=2}^c \nu(H_i) + \nu(H_1 \setminus \Gamma_{H_1}(C_m)).
		$$
		\item For each $i=1,\ldots,c$, we have that $\lvert F_i \cap C_n \rvert=1$.
		\item The statement of \autoref{components_case_only_one_bad} also holds in this case, that is, if $\nu(H_i)=\nu(F_i)$ for all $1 \le i \le c$, then $\nu(G \setminus \Gamma_G(C_n))=\nu(G)$.
		\item  	Due to the assumption $l\ge 3$, then we have that $C_m$ must be an induced subgraph of $H_1$.
		During this subsection and the next one we shall fundamentally use this fact, and it will allow us to inductively ``separate''  the two cycles $C_n$ and $C_m$.
	\end{enumerate}
\end{rem}

\begin{lem}
	\label{implication_trees_nu_gamma}
	Adopt \autoref{notation_trees_in_bicyclic}.
	If $\nu(H_i)=\nu(F_i)$ for all $1\le i \le c$ and $\nu(H_1)=\nu(H_1 \setminus \Gamma_{H_1}(C_m))$, then
	$$
	\nu\left(G\setminus \Gamma_G(C_n\cup C_m)\right) = \nu(G).
	$$
	\begin{proof}
		Since $G \setminus \Gamma_G(C_n\cup C_m)$ is an induced subgraph of $G$, then we have $\nu(G \setminus \Gamma_G(C_n\cup C_m)) \le \nu(G)$.
		From \autoref{rem_Case_II}$(ii)$ we get
		\begin{align*}
		\nu(G \setminus \Gamma_G(C_n\cup C_m)) &= \nu(C_n) +  \sum_{i=2}^c \nu(H_i) + \nu(H_1 \setminus \Gamma_{H_1}(C_m))\\
		&= \nu(C_n) +  \sum_{i=2}^c \nu(H_i) + \nu(H_1)\\
		&= \nu(C_n) + \sum_{i=1}^c \nu(F_i)\\
		&\ge \nu(G),
		\end{align*} 
		and so $\nu\left(G\setminus \Gamma_G(C_n\cup C_m)\right) = \nu(G)$.
	\end{proof}
\end{lem}

\begin{prop}
	\label{implication_two_bad_cycles_reg_3}
	Adopt \autoref{notation_trees_in_bicyclic}.
	If $\nu(G\setminus \Gamma_G(C_n\cup C_m)) < \nu(G)$, then 
	$$
	\reg{I(G)} \le \nu(G) + 2.
	$$
	\begin{proof}
		It follows from the contrapositive of \autoref{implication_trees_nu_gamma}, that there exists some $i$ with $\nu(H_i) < \nu(F_i)$ or we have $\nu(H_1 \setminus \Gamma_{H_1}(C_m)) < \nu(H_1)$.
		Then we divide the proof into two cases.
		
		\underline{Case 1.} In this case we assume that for some $1\le  i \le c$ we have $\nu(H_i) < \nu(F_i)$.
		This case follows similarly to \autoref{implic_one_bad_cycle}.		
		Let $x$ be the vertex in $F_i \cap C_n$, let us use the notations $G^\prime=G\setminus x$ and $G^{\prime\prime} =G \setminus N[x]$.
		Once more, we have the inequality 
		$$
		\reg{I(G)} \le \max\{\reg{I(G^\prime)}, \reg{I(G^{\prime\prime})}+1 \}.
		$$  
		Note that both $G^\prime$ and $G^{\prime\prime}$ are unicyclic graphs, and so we have $\reg{I(G^\prime)} \le \nu(G^\prime) + 2$ and $\reg{I(G^{\prime\prime})} \le \nu(G^{\prime\prime}) + 2$ (see \autoref{Upper_bound_reg_decyc}).
		Since we have $\nu(G^{\prime}) \le \nu(G)$ and $\nu(G^{\prime\prime})+1\le \nu(G)$ (see the proof of \autoref{implic_one_bad_cycle}), then the inequality follows in this case.
		
		\underline{Case 2.}	Now we suppose that $\nu(H_1 \setminus \Gamma_{H_1}(C_m)) < \nu(H_1)$.	
		Let $x$ be the vertex in $F_1 \cap C_n$, let us use the notations $G^\prime=G\setminus x$ and $G^{\prime\prime} =G \setminus N[x]$.
		We use the inequality 
		$$
		\reg{I(G)} \le \max\{\reg{I(G^\prime)}, \reg{I(G^{\prime\prime})}+1 \}.
		$$ 
		The graphs $G^{\prime}$ and $G^{\prime\prime}$ are unicyclic. 
		For the graph $G^\prime$ we have $\reg{I(G^{\prime})} \le \nu(G^{\prime})+2 \le \nu(G) + 2$.
		The graph $G^{\prime\prime}$ can be given as the disjoint union of $H_1$ and another graph  $H$ defined by $H=G \setminus (F_1 \cup N[x])$, that is $G^{\prime\prime}=H \cup H_1$ and $H \cap H_1 = \emptyset$.
		Since $H$ is a forest, then using \cite[Theorem 1.1]{REG_UNICYCLIC_GRAPH} we obtain that $\reg{I(G^{\prime\prime})} \le \nu(G^{\prime\prime}) + 1$.
		So we get the  inequality $\reg{I(G^{\prime\prime})}+1 \le \nu(G^{\prime\prime}) +2 \le \nu(G) +2$, because $G
		^{\prime\prime}$ is an induced subgraph of $G$.
	\end{proof}
\end{prop}

Now we are ready to completely describe the case where $\reg{I(G)}=\nu(G)+3$.

\begin{theorem}
	\label{thm_case_nu+3}
	Let $G$ be a bicyclic graph with dumbbell $\cnplcm$.
	Then $\reg{I(G)}=\nu(G)+3$ if and only if the following conditions are satisfied:
	\begin{enumerate}[(i)]
		\item $n \equiv 2 \; (\text{\normalfont mod}\;3)$;
		\item $m \equiv 2 \; (\text{\normalfont mod}\;3)$;
		\item $l \ge 3$;
		\item $\nu\left(G\setminus \Gamma_G(C_n\cup C_m)\right) = \nu(G)$.
	\end{enumerate}
	\begin{proof}
		In \autoref{general_prop_bicyclic} we proved that the conditions $(i)$, $(ii)$ and $(iii)$ are necessary, and from \autoref{implication_two_bad_cycles_reg_3} we have that the condition $(iv)$ is also necessary.
		Hence, we only need to prove that $\reg{I(G)}=\nu(G)+3$ under these conditions.
		
		Let $W=G\setminus \Gamma_G(C_n\cup C_m)$.
		From \autoref{rem_Case_II}, and using \cite[Theorem 1.1]{REG_UNICYCLIC_GRAPH} and \autoref{reg_Forests}, we can compute
		\begin{align*}
		\reg{\big(I(W)\big)} &= \reg{\big(I(C_n)\big)} + \reg{\big(I\big(\cup_{i=2}^c H_i \big)\big)}  + \reg{\big(I\big( H_1 \setminus \Gamma_{H_1}(C_m) \big)\big)} - 2\\
		&=(\nu(C_n)+2) + (\nu(\cup_{i=2}^c H_i) + 1) + (\nu(H_1 \setminus \Gamma_{H_1}(C_m)) + 2) - 2\\
		&=\nu(W) + 3\\
		&=\nu(G) + 3.
		\end{align*} 
		Since $W$ is an induced subgraph of $G$ then we get 
		$$
		\reg{I(G)} \ge \reg{I(W))}=\nu(G)+3,
		$$
		and so from \autoref{Upper_bound_reg_decyc} the equality it is obtained.
	\end{proof}
\end{theorem}

\subsection{Case III}\hspace{\fill}\\
In this subsection we assume that $G$ is a bicyclic graph with dumbbell $\cnplcm$ such that $n,m \equiv 2 \; (\text{mod}\;3)$ and $l \ge 3$. 
Now that we have characterized when $\reg{I(G)}=\nu(G)+3$, then we want to distinguish between $\reg{I(G)}=\nu(G)+1$ and $\reg{I(G)}=\nu(G)+2$.

\begin{lem}
	\label{lem_reg_nu+2_both_cycles}
	Adopt \autoref{notation_trees_in_bicyclic}.
	If $\nu(G) - \nu\left(G\setminus \Gamma_G(C_n\cup C_m)\right) = 1$ then
	$$
	\reg{I(G)} = \nu(G) + 2.
	$$ 
	\begin{proof}
		From \autoref{thm_case_nu+3} we have that $\reg(I(G)) \le \nu(G) + 2$.
		Using the same method as in \autoref{thm_case_nu+3}, we can obtain a lower bound 
		$$
		\reg{I(G)}\ge \reg{I(G\setminus \Gamma_G(C_n\cup C_m))} = \nu(G\setminus \Gamma_G(C_n\cup C_m)) + 3 = \nu(G) + 2,
		$$
		and so the equality follows.
	\end{proof}
\end{lem}

\begin{lem}
	\label{lem_reg_nu+2_one_cycle}
	Adopt \autoref{notation_trees_in_bicyclic}.
	If $\nu(G)=\nu(G\setminus \Gamma_G(C_n))$ then 
	$$
	\reg{I(G)} \ge \nu(G) + 2.
	$$
	Symmetrically, the same argument holds for $C_m$.
	\begin{proof}
		The proof follows similarly to \autoref{thm_only_one_bad_cycle}.
		From \autoref{rem_Case_II}$(i)$,  \cite[Theorem 1.1]{REG_UNICYCLIC_GRAPH} and \autoref{reg_Forests} we get
		\begin{align*}
		\reg{I(G\setminus \Gamma_G(C_n))} &= \reg{I(C_n)} + \reg{I(\cup_{i=1}^cH_i)} - 1\\
		&= (\nu(C_n) + 2) +(\nu(\cup_{i=1}^cH_i) + 1) -1 \\
		&=\nu(G \setminus \Gamma_G(C_n)) + 2\\
		&=\nu(G) + 2.
		\end{align*}
		So the inequality follows from the fact that $G\setminus \Gamma_G(C_n)$ is an induced subgraph of $G$.
	\end{proof}
\end{lem}

The following very simple logical argument will be used several times in the next theorem.

\begin{obs}
	\label{obs_Logic}		
	Let $P_1,P_2,P_3$ be boolean values, (i.e. true or false). 
	Assume that $P_1$ is true if and only if $P_2$ and $P_3$ are true, that is
	$$
	P_1 \;\Longleftrightarrow\; \left( P_2 \quad \wedge \quad P_3\right).
	$$ 
	Suppose that if $P_2$ is true then $P_3$ is false, that is
	$$
	P_2 \Longrightarrow \lnot P_3.
	$$
	Then, $P_1$ is false.	
\end{obs}

\begin{nota}
	Let $X$ be a mathematical expression. 
	Then, $P[X]$ represents a boolean value, which is true if $X$ is satisfied and false otherwise.
\end{nota}

Taking into account the induced matching numbers $\nu(G)$, $\nu(G\setminus \Gamma_G(C_n\cup C_m))$, $\nu(G\setminus \Gamma_G(C_n))$ and $\nu(G \setminus \Gamma_G(C_m))$, we can give necessary and sufficient conditions for the equality $\reg{I(G)} = \nu(G) + 1$. 

\begin{theorem}
	\label{charact_thm_nu+1_in_bad_case}
	Let $G$ be a bicyclic graph with dumbbell $\cnplcm$ such that $n,m \equiv 2 \; (\Mod 3)$ and $l \ge 3$. 
	Then $\reg{I(G)}=\nu(G)+1$ if and only if the following conditions are satisfied: 
	\begin{enumerate}[(i)]
		\item $\nu(G)-\nu(G\setminus \Gamma_G(C_n\cup C_m)) > 1$;
		\item $\nu(G) > \nu(G \setminus \Gamma_G(C_n))$;
		\item $\nu(G) > \nu(G \setminus \Gamma_G(C_m))$.
	\end{enumerate}
	\begin{proof}
		From \autoref{lem_reg_nu+2_both_cycles} and \autoref{lem_reg_nu+2_one_cycle}, we have that the conditions $(i)$, $(ii)$ and $(iii)$ are necessary. 
		Hence, it is enough to prove $\reg{I(G)} \le \nu(G)+1$ under these conditions.
		
		Again, for any $x \in G$ we denote $G^\prime =G \setminus x$ and $G^{\prime\prime}=G\setminus N[x]$. 
		We have the upper bound 
		$$
		\reg{I(G)} \le \max\{\reg{I(G^\prime)}, \reg{I(G^{\prime\prime})}+1 \}.
		$$
		We shall prove that under the conditions $(i)$, $(ii)$ and $(iii)$ there exists a vertex $x \in C_n$ such that $\reg{I(G^\prime)} \le \nu(G)+ 1$ and $\reg{I(G^{\prime\prime})}+1 \le \nu(G)+1$.
		We divide the proof into three steps.
		
		\underline{Step 1.} In this step we prove that for any $x \in C_n$ we have $\reg{I(G^\prime)} \le \nu(G) + 1$.
		First we note the following two statements: 
		\begin{compactitem}
			\item From \autoref{Upper_bound_reg_decyc} we have that $\reg{I(G^\prime)} \le \nu(G^\prime)+2$.
			Hence, $\nu(G^\prime) < \nu(G)$ implies that $\reg{I(G^\prime)} \le \nu(G^\prime)+2 \le \nu(G) + 1$.
			\item From \cite[Theorem 1.1]{REG_UNICYCLIC_GRAPH} we obtain that $\reg{I(G^\prime)} = \nu(G^\prime) + 2$ if and only if $\nu(G^\prime) = \nu(G^\prime \setminus \Gamma_{G^\prime}(C_m))$.
		\end{compactitem}
		Thus, it follows that
		$$
		\reg{I(G^\prime)} = \nu(G)+2 \;\Longleftrightarrow\; \Big(\nu(G) = \nu(G^\prime) \text{ and } \nu(G^\prime) =  \nu(G^\prime \setminus \Gamma_{G^\prime}(C_m))\Big).
		$$
		
		In \autoref{obs_Logic}, let $P_1=P\big[		\reg{I(G^\prime)} = \nu(G)+2\big]$, $P_2=P\big[\nu(G) = \nu(G^\prime)\big]$ and $P_3=\big[\nu(G^\prime) =  \nu(G^\prime \setminus \Gamma_{G^\prime}(C_m))\big]$.
		From the logical argument of \autoref{obs_Logic}, if we prove that $\nu(G^\prime)=\nu(G)$ implies $\nu(G^\prime) > \nu(G^\prime \setminus \Gamma_{G^\prime}(C_m))$ then we will get the desired inequality $\reg{I(G^\prime)} \le \nu(G) +1$.
		Assume that $\nu(G)=\nu(G^\prime)$. From the hypothesis $\nu(G) > \nu(G \setminus \Gamma_{G}(C_m))$ and the fact that $G^\prime \setminus \Gamma_{G^\prime}(C_m)$ is an induced subgraph of $G \setminus \Gamma_{G}(C_m)$, then we get 
		$$
		\nu(G^\prime)=\nu(G) > \nu(G \setminus \Gamma_{G}(C_m)) \ge \nu(G^\prime \setminus \Gamma_{G^\prime}(C_m)).
		$$
		Therefore, we have $\reg{I(G^\prime)} \le \nu(G) +1$.
			
		\underline{Step 2.} Since $\nu(G) > \nu(G\setminus\Gamma_{G}(C_n))$, it follows from \autoref{rem_Case_II}$(iv)$ that there exists some $1 \le i \le c$ such that $\nu(F_i) > \nu(H_i)$.
		Following \autoref{notation_trees_in_bicyclic}, we have that $F_1$ is a unicyclic graph containing the cycle $C_m$ and that $F_i$ is a tree for all $i > 1$.
		In this step, fix $i > 1$ where $F_i$ is a tree and $\nu(F_i)>\nu(H_i)$.
		
		Let $x$ be the vertex in $F_i \cap C_n$ and $H$ be the induced subgraph $H = G \setminus (F_i \cup N_G[x])$.
		Note that $G^{\prime\prime}=H \cup H_i$, $d(H,H_i)\ge 2$ and $d(H,F_i)\ge 2$. 
		Then
		$$
		\nu(G^{\prime\prime}) = \nu(H) + \nu(H_i) < \nu(H) + \nu(F_i) \le \nu(G) 
		$$
		follows from the condition $\nu(H_i) < \nu(F_i)$.
		So we have that $\nu(G^{\prime\prime})< \nu(G)$.
		
		Let $K$ be the induced subgraph defined by $K=(G \setminus \Gamma_{G}(C_m)) \setminus(F_i \cup N[x])$.
		Since $i>1$ then $F_i \cap F_1=\emptyset$, and so we get the following statements:
		\begin{compactitem}
			\item $G^{\prime\prime} \setminus \Gamma_{G^{\prime\prime}}(C_m)=K \cup H_i$.
			\item $K\cup F_i$ is an induced subgraph of $G \setminus \Gamma_{G}(C_m)$.
			\item We have the following inequalities 
			$$
			\nu(G^{\prime\prime} \setminus \Gamma_{G^{\prime\prime}}(C_m))=\nu(K)+\nu(H_i)<\nu(K)+\nu(F_i)\le \nu(G \setminus \Gamma_{G}(C_m)).
			$$ 
		\end{compactitem}
		
		Again, as in \underline{Step 1}, \cite[Theorem 1.1]{REG_UNICYCLIC_GRAPH} and \autoref{Upper_bound_reg_decyc} yield the following equivalence
		$$
		\reg{I(G^{\prime\prime})}+1 = \nu(G)+2 \;\Longleftrightarrow\; \Big(\nu(G) =  \nu(G^{\prime\prime})+1 \text{ and } \nu(G^{\prime\prime}) =  \nu(G^{\prime\prime} \setminus \Gamma_{G^{\prime\prime}}(C_m))\Big).
		$$
		In \autoref{obs_Logic}, let $P_1=P\big[		\reg{I(G^{\prime\prime})}+1 = \nu(G)+2\big]$, $P_2=P\big[\nu(G) = \nu(G^{\prime\prime}+1)\big]$ and $P_3=\big[\nu(G^{\prime\prime}) =  \nu(G^{\prime\prime} \setminus \Gamma_{G^\prime}(C_m))\big]$.
		So it is enough to prove that $\nu(G)=\nu(G^{\prime\prime})+1$ implies $\nu(G^{\prime\prime}) >  \nu(G^{\prime\prime} \setminus \Gamma_{G^{\prime\prime}}(C_m))$.
		Assuming $\nu(G)=\nu(G^{\prime\prime})+1$ then we get 
		$$
		\nu(G^{\prime\prime})=\nu(G)-1>\nu(G \setminus \Gamma_G(C_m))-1\ge \nu(G^{\prime\prime} \setminus \Gamma_{G^{\prime\prime}}(C_m)).
		$$
		Therefore, in this case we have $\reg{I(G^{\prime\prime})}+1 \le \nu(G)+1$.
		
		\underline{Step 3.}
		In this last step we assume that $\nu(F_1) > \nu(H_1)$ and that $\nu(F_i)=\nu(H_i)$ for all $i>1$.
		Let $x$ be the vertex in $F_1 \cap C_n$, then as in \underline{Step 2} we have the statements:
		\begin{compactitem}
			\item $\nu(G^{\prime\prime})<\nu(G)$.
			\item $\reg{I(G^{\prime\prime})}+1 = \nu(G)+2 \;\Longleftrightarrow\; \Big(\nu(G) = \nu(G^{\prime\prime})+1 \text{ and } \nu(G^{\prime\prime}) =  \nu(G^{\prime\prime} \setminus \Gamma_{G^{\prime\prime}}(C_m))\Big).$
		\end{compactitem}
		Once more, if we prove that $\nu(G)=\nu(G^{\prime\prime})+1$ implies $\nu(G^{\prime\prime}) >  \nu(G^{\prime\prime} \setminus \Gamma_{G^{\prime\prime}}(C_m))$ then we obtain that $\reg{I(G^{\prime\prime})}+1 \le \nu(G)+1$.
		
		We denote by $L$ the induced subgraph of $G^{\prime\prime} \setminus \Gamma_{G^{\prime\prime}}(C_m)$ given by disconnecting all the trees $F_i$ with $i > 1$, that is  
		$$
		L = (G^{\prime\prime} \setminus \Gamma_{G^{\prime\prime}}(C_m)) \setminus \Gamma_{G}(C_n).
		$$
		From the conditions $\nu(F_i)=\nu(H_i)$ for all $i>1$, then we get $\nu(L)=\nu(G^{\prime\prime} \setminus \Gamma_{G^{\prime\prime}}(C_m))$ (see the proofs of \autoref{components_case_only_one_bad} or \autoref{implication_trees_nu_gamma}).
		We also have that $L$ is an induced subgraph of $G \setminus \Gamma_{G}(C_n\cup C_m)$ because we have the equality 
		$$
		L = (G \setminus \Gamma_{G}(C_n\cup C_m)) \setminus N[x].
		$$
		Finally, from the hypothesis $\nu(G)-\nu(G\setminus \Gamma_G(C_n\cup C_m)) > 1$ we can obtain
		$$
		\nu(G^{\prime\prime})=\nu(G)-1>\nu(G\setminus \Gamma_G(C_n\cup C_m))\ge \nu(L) = \nu(G^{\prime\prime} \setminus \Gamma_{G^{\prime\prime}}(C_m)).
		$$
		Therefore, in this case we also have $\reg{I(G^{\prime\prime})}+1 \le \nu(G)+1$.
	\end{proof}
\end{theorem}

\subsection{Case IV}\hspace{\fill}\\
In this short subsection we deal with the remaining case, we assume that $G$ is a bicyclic graph with dumbbell $\cnplcm$ such that $n,m \equiv 2 \; (\text{mod}\;3)$ and $l \le 2$.

When $l\le 2$, the two cycles are too close to each other, and it is difficult to make a direct analysis (with our methods). 
Fortunately, using the complete characterization of the case $l\ge 3$, the problem can be solved with the Lozin transformation.
Suppose that $x$ is a vertex on the bridge $P_l$ (at most two), then we apply the Lozin transformation of $G$ with respect to $x$, and obtain a bicyclic graph $\LozGraph$ with dumbbell of the type $C_n\cdot P_k\cdot C_m$ where $k \ge 4$.
From \cite[Lemma 1]{LOZIN_TRANSFORMATION} and \cite[Theorem 1.1]{LOZIN_TRANS} we get the equality 
\begin{equation}
\label{enlage_with_Lozin}
\reg{\left(I(\LozGraph)\right)} - \nu\left(\LozGraph\right) = \reg{\left(I(G)\right)} - \nu\left(G\right).
\end{equation}

Therefore we obtain a characterization in the following corollary.

\begin{cor}
	\label{cor_case_l<=2}
	Let $G$ be a bicyclic graph with dumbbell $\cnplcm$ such that $n,m \equiv 2 \; (\Mod 3)$ and $l \le 2$. 
	Let $x$ be a point on the bridge $P_l$ and let $\LozGraph$ be the Lozin transformation of $G$ with respect to $x$.
	Then we have that $\nu(G)+1\le\reg{I(G)}\le\nu(G)+2$, and that $\reg{I(G)}=\nu(G)+1$ if and only if the following conditions are satisfied: 
	\begin{enumerate}[(i)]
		\item $\nu(\LozGraph)-\nu(\LozGraph\setminus \Gamma_{\LozGraph}(C_n\cup C_m)) > 1$;
		\item $\nu(\LozGraph) > \nu(\LozGraph \setminus \Gamma_{\LozGraph}(C_n))$;
		\item $\nu(\LozGraph) > \nu(\LozGraph \setminus \Gamma_{\LozGraph}(C_m))$.
	\end{enumerate}
	\begin{proof}
		It follows from \autoref{general_prop_bicyclic}, \autoref{enlage_with_Lozin}, and \autoref{charact_thm_nu+1_in_bad_case}. 
	\end{proof}
\end{cor}

%\begin{rem}
%	The previous conditions given in terms of the graph $\LozGraph$ can be read easily from the original graph $G$.
%	We preferred to give a neat and recursive notation, than a larger notation in terms of the original graph.
%\end{rem}

\subsection{Examples}\hspace{\fill}\\
In this last subsection we shall give examples for each one of the statements in the characterization of \autoref{full_characterization}.

\begin{exmp}
	Statement $\textit{(I)}$ of \autoref{full_characterization}. 
	Let $G$ be the graph below.
	\begin{center}
		\definecolor{ududff}{rgb}{0.30196078431372547,0.30196078431372547,1.}
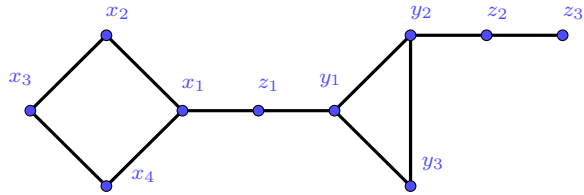
		\begin{tikzpicture}[line cap=round,line join=round,>=triangle 45,x=1.0cm,y=1.0cm]
		\clip(3.5,2.5) rectangle (12.,5.5);
		\draw [line width=1.2pt] (4.,4.)-- (5.,5.);
		\draw [line width=1.2pt] (5.,5.)-- (6.,4.);
		\draw [line width=1.2pt] (6.,4.)-- (5.,3.);
		\draw [line width=1.2pt] (5.,3.)-- (4.,4.);
		\draw [line width=1.2pt] (6.,4.)-- (7.,4.);
		\draw [line width=1.2pt] (7.,4.)-- (8.,4.);
		\draw [line width=1.2pt] (8.,4.)-- (9.,5.);
		\draw [line width=1.2pt] (9.,5.)-- (9.,3.);
		\draw [line width=1.2pt] (9.,3.)-- (8.,4.);
		\draw [line width=1.2pt] (9.,5.)-- (10.,5.);
		\draw [line width=1.2pt] (10.,5.)-- (11.,5.);
		\begin{scriptsize}
		\draw [fill=ududff] (4.,4.) circle (2.0pt);
		\draw[color=ududff] (3.8736235633288434,4.421059619450316) node {$x_3$};
		\draw [fill=ududff] (5.,5.) circle (2.0pt);
		\draw[color=ududff] (5.149810378627708,5.297477070920621) node {$x_2$};
		\draw [fill=ududff] (6.,4.) circle (2.0pt);
		\draw[color=ududff] (6.1492337881991075,4.344180895637131) node {$x_1$};
		\draw [fill=ududff] (5.,3.) circle (2.0pt);
		\draw[color=ududff] (5.488192081491441,3.1294970593888136) node {$x_4$};
		\draw [fill=ududff] (7.,4.) circle (2.0pt);
		\draw[color=ududff] (7.13328145300787,4.344180895637131) node {$z_1$};
		\draw [fill=ududff] (8.,4.) circle (2.0pt);
		\draw[color=ududff] (7.948195925427626,4.4056838746876785) node {$y_1$};
		\draw [fill=ududff] (9.,5.) circle (2.0pt);
		\draw[color=ududff] (9.147504016913304,5.297477070920621) node {$y_2$};
		\draw [fill=ududff] (9.,3.) circle (2.0pt);
		\draw[color=ududff] (9.297504016913304,3.2986302517778197) node {$y_3$};
		\draw [fill=ududff] (10.,5.) circle (2.0pt);
		\draw[color=ududff] (10.146927426484703,5.297477070920621) node {$z_2$};
		\draw [fill=ududff] (11.,5.) circle (2.0pt);
		\draw[color=ududff] (11.146350836056103,5.297477070920621) node {$z_3$};
		\end{scriptsize}
		\end{tikzpicture}
	\end{center}
	Then we have $\reg{I(G)}=4$ and $\nu(G)=3$.
\end{exmp}

\begin{exmp}
	Statement $\textit{(II)}$ of \autoref{full_characterization}.
	Let $G$ be the graph below.
	\begin{center}
		\definecolor{ududff}{rgb}{0.30196078431372547,0.30196078431372547,1.}
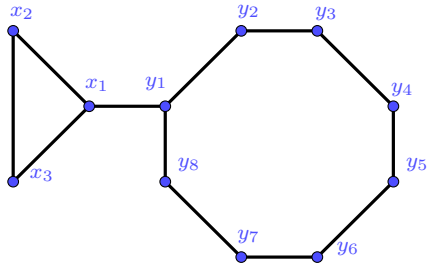
		\begin{tikzpicture}[line cap=round,line join=round,>=triangle 45,x=1.0cm,y=1.0cm]
		\clip(1.5,1.5) rectangle (7.5,5.5);
		\draw [line width=1.2pt] (2.,5.)-- (2.,3.);
		\draw [line width=1.2pt] (2.,3.)-- (3.,4.);
		\draw [line width=1.2pt] (3.,4.)-- (2.,5.);
		\draw [line width=1.2pt] (3.,4.)-- (4.,4.);
		\draw [line width=1.2pt] (4.,4.)-- (5.,5.);
		\draw [line width=1.2pt] (5.,5.)-- (6.,5.);
		\draw [line width=1.2pt] (6.,5.)-- (7.,4.);
		\draw [line width=1.2pt] (7.,4.)-- (7.,3.);
		\draw [line width=1.2pt] (7.,3.)-- (6.,2.);
		\draw [line width=1.2pt] (6.,2.)-- (5.,2.);
		\draw [line width=1.2pt] (5.,2.)-- (4.,3.);
		\draw [line width=1.2pt] (4.,3.)-- (4.,4.);
		\begin{scriptsize}
		\draw [fill=ududff] (2.,5.) circle (2.0pt);
		\draw[color=ududff] (2.1170466559612455,5.233846923822523) node {$x_2$};
		\draw [fill=ududff] (2.,3.) circle (2.0pt);
		\draw[color=ududff] (2.377144873844442,3.094350367180728) node {$x_3$};
		\draw [fill=ududff] (3.,4.) circle (2.0pt);
		\draw[color=ududff] (3.102657429245669,4.284295081268019) node {$x_1$};
		\draw [fill=ududff] (4.,4.) circle (2.0pt);
		\draw[color=ududff] (3.883934261727225,4.308334368421297) node {$y_1$};
		\draw [fill=ududff] (5.,5.) circle (2.0pt);
		\draw[color=ududff] (5.097918262967795,5.257886210975802) node {$y_2$};
		\draw [fill=ududff] (6.,5.) circle (2.0pt);
		\draw[color=ududff] (6.119587966982137,5.233846923822523) node {$y_3$};
		\draw [fill=ududff] (7.,4.) circle (2.0pt);
		\draw[color=ududff] (7.117218383843199,4.224196863384822) node {$y_4$};
		\draw [fill=ududff] (7.,3.) circle (2.0pt);
		\draw[color=ududff] (7.317218383843199,3.2265664465237607) node {$y_5$};
		\draw [fill=ududff] (6.,2.) circle (2.0pt);
		\draw[color=ududff] (6.403725472018612,2.108739593896306) node {$y_6$};
		\draw [fill=ududff] (5.,2.) circle (2.0pt);
		\draw[color=ududff] (5.097918262967795,2.277014603969256) node {$y_7$};
		\draw [fill=ududff] (4.,3.) circle (2.0pt);
		\draw[color=ududff] (4.312307489683372,3.2265664465237607) node {$y_8$};
		\end{scriptsize}
		\end{tikzpicture}
	\end{center}
	Then we have $\reg{I(G)}=5$ and $\nu(G)=3$.
	
	On the other hand, let $G$ be the graph below.
	\begin{center}
		\definecolor{ududff}{rgb}{0.30196078431372547,0.30196078431372547,1.}
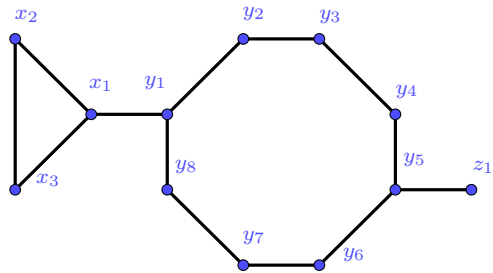
		\begin{tikzpicture}[line cap=round,line join=round,>=triangle 45,x=1.0cm,y=1.0cm]
		\clip(1.5,1.5) rectangle (8.5,5.5);
		\draw [line width=1.2pt] (2.,5.)-- (2.,3.);
		\draw [line width=1.2pt] (2.,3.)-- (3.,4.);
		\draw [line width=1.2pt] (3.,4.)-- (2.,5.);
		\draw [line width=1.2pt] (3.,4.)-- (4.,4.);
		\draw [line width=1.2pt] (4.,4.)-- (5.,5.);
		\draw [line width=1.2pt] (5.,5.)-- (6.,5.);
		\draw [line width=1.2pt] (6.,5.)-- (7.,4.);
		\draw [line width=1.2pt] (7.,4.)-- (7.,3.);
		\draw [line width=1.2pt] (7.,3.)-- (6.,2.);
		\draw [line width=1.2pt] (6.,2.)-- (5.,2.);
		\draw [line width=1.2pt] (5.,2.)-- (4.,3.);
		\draw [line width=1.2pt] (4.,3.)-- (4.,4.);
		\draw [line width=1.2pt] (7.,3.)-- (8.,3.);
		\begin{scriptsize}
		\draw [fill=ududff] (2.,5.) circle (2.0pt);
		\draw[color=ududff] (2.1543604565701857,5.29946633117295) node {$x_2$};
		\draw [fill=ududff] (2.,3.) circle (2.0pt);
		\draw[color=ududff] (2.4342150687564937,3.127420879705376) node {$x_3$};
		\draw [fill=ududff] (3.,4.) circle (2.0pt);
		\draw[color=ududff] (3.1285867252431423,4.3891237522490405) node {$x_1$};
		\draw [fill=ududff] (4.,4.) circle (2.0pt);
		\draw[color=ududff] (3.847278234919914,4.4210655971235635) node {$y_1$};
		\draw [fill=ududff] (5.,5.) circle (2.0pt);
		\draw[color=ududff] (5.140922952338102,5.331408176047473) node {$y_2$};
		\draw [fill=ududff] (6.,5.) circle (2.0pt);
		\draw[color=ududff] (6.147091065885583,5.29946633117295) node {$y_3$};
		\draw [fill=ududff] (7.,4.) circle (2.0pt);
		\draw[color=ududff] (7.153259179433062,4.3092691400627325) node {$y_4$};
		\draw [fill=ududff] (7.,3.) circle (2.0pt);
		\draw[color=ududff] (7.253259179433062,3.398926561138823) node {$y_5$};
		\draw [fill=ududff] (6.,2.) circle (2.0pt);
		\draw[color=ududff] (6.458887522946414,2.137223688595158) node {$y_6$};
		\draw [fill=ududff] (5.,2.) circle (2.0pt);
		\draw[color=ududff] (5.140922952338102,2.3608166027168203) node {$y_7$};
		\draw [fill=ududff] (4.,3.) circle (2.0pt);
		\draw[color=ududff] (4.250725761227884,3.3031010265152534) node {$y_8$};
		\draw [fill=ududff] (8.,3.) circle (2.0pt);
		\draw[color=ududff] (8.159427292980542,3.3031010265152534) node {$z_1$};
		\end{scriptsize}
		\end{tikzpicture}
	\end{center}
	Then we have $\reg{I(G)}=5$ and $\nu(G)=4$.
\end{exmp}

\begin{exmp}
	Statement $\textit{(III)}$ of \autoref{full_characterization}.
	In \autoref{examp_reg_nu+3} we saw a graph $G$ where $\reg{I(G)}=6$ and $\nu(G)=3$.
	
	Let $G$ be the graph below.
	\begin{center}
		\definecolor{ududff}{rgb}{0.30196078431372547,0.30196078431372547,1.}
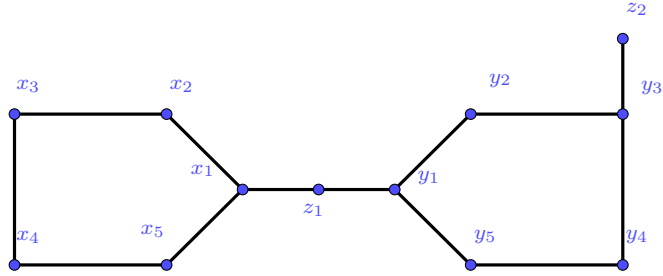
		\begin{tikzpicture}[line cap=round,line join=round,>=triangle 45,x=1.0cm,y=1.0cm]
		\clip(1.5,1.5) rectangle (10.7,5.5);
		\draw [line width=1.2pt] (2.,4.)-- (2.,2.);
		\draw [line width=1.2pt] (2.,2.)-- (4.,2.);
		\draw [line width=1.2pt] (4.,2.)-- (5.,3.);
		\draw [line width=1.2pt] (5.,3.)-- (4.,4.);
		\draw [line width=1.2pt] (4.,4.)-- (2.,4.);
		\draw [line width=1.2pt] (5.,3.)-- (6.,3.);
		\draw [line width=1.2pt] (6.,3.)-- (7.,3.);
		\draw [line width=1.2pt] (7.,3.)-- (8.,4.);
		\draw [line width=1.2pt] (8.,4.)-- (10.,4.);
		\draw [line width=1.2pt] (10.,4.)-- (10.,2.);
		\draw [line width=1.2pt] (10.,2.)-- (8.,2.);
		\draw [line width=1.2pt] (8.,2.)-- (7.,3.);
		\draw [line width=1.2pt] (10.,4.)-- (10.,5.);
		\begin{scriptsize}
		\draw [fill=ududff] (2.,4.) circle (2.0pt);
		\draw[color=ududff] (2.18146341463415,4.371060692002271) node {$x_3$};
		\draw [fill=ududff] (2.,2.) circle (2.0pt);
		\draw[color=ududff] (2.18146341463415,2.3744639818491238) node {$x_4$};
		\draw [fill=ududff] (4.,2.) circle (2.0pt);
		\draw[color=ududff] (3.818672716959731,2.4543278502552495) node {$x_5$};
		\draw [fill=ududff] (5.,3.) circle (2.0pt);
		\draw[color=ududff] (4.4775496313102705,3.2529665343165086) node {$x_1$};
		\draw [fill=ududff] (4.,4.) circle (2.0pt);
		\draw[color=ududff] (4.1980260918888295,4.371060692002271) node {$x_2$};
		\draw [fill=ududff] (6.,3.) circle (2.0pt);
		\draw[color=ududff] (5.935065229722069,2.73385138967669) node {$z_1$};
		\draw [fill=ududff] (7.,3.) circle (2.0pt);
		\draw[color=ududff] (7.452648893930804,3.1529665343165086) node {$y_1$};
		\draw [fill=ududff] (8.,4.) circle (2.0pt);
		\draw[color=ududff] (8.39087918321044,4.450924560408398) node {$y_2$};
		\draw [fill=ududff] (10.,4.) circle (2.0pt);
		\draw[color=ududff] (10.387475893363588,4.35109472490074) node {$y_3$};
		\draw [fill=ududff] (10.,2.) circle (2.0pt);
		\draw[color=ududff] (10.187816222348273,2.3744639818491238) node {$y_4$};
		\draw [fill=ududff] (8.,2.) circle (2.0pt);
		\draw[color=ududff] (8.191219512195126,2.3744639818491238) node {$y_5$};
		\draw [fill=ududff] (10.,5.) circle (2.0pt);
		\draw[color=ududff] (10.187816222348273,5.389325014180377) node {$z_2$};
		\end{scriptsize}
		\end{tikzpicture}
	\end{center}
	Then we have $\reg{I(G)}=5$ and $\nu(G)=3$.  
	
	But if we move the outer edge to the left, then we get a different result. 
	Let $G$ be the graph below. 
	\begin{center}
		\definecolor{ududff}{rgb}{0.30196078431372547,0.30196078431372547,1.}
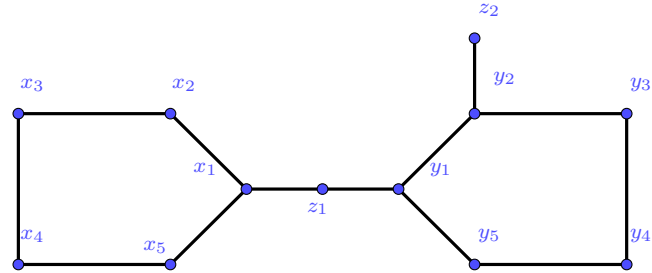
		\begin{tikzpicture}[line cap=round,line join=round,>=triangle 45,x=1.0cm,y=1.0cm]
		\clip(1.5,1.5) rectangle (10.6,5.5);
		\draw [line width=1.2pt] (2.,4.)-- (2.,2.);
		\draw [line width=1.2pt] (2.,2.)-- (4.,2.);
		\draw [line width=1.2pt] (4.,2.)-- (5.,3.);
		\draw [line width=1.2pt] (5.,3.)-- (4.,4.);
		\draw [line width=1.2pt] (4.,4.)-- (2.,4.);
		\draw [line width=1.2pt] (5.,3.)-- (6.,3.);
		\draw [line width=1.2pt] (6.,3.)-- (7.,3.);
		\draw [line width=1.2pt] (7.,3.)-- (8.,4.);
		\draw [line width=1.2pt] (8.,4.)-- (10.,4.);
		\draw [line width=1.2pt] (10.,4.)-- (10.,2.);
		\draw [line width=1.2pt] (10.,2.)-- (8.,2.);
		\draw [line width=1.2pt] (8.,2.)-- (7.,3.);
		\draw [line width=1.2pt] (8.,4.)-- (8.,5.);
		\begin{scriptsize}
		\draw [fill=ududff] (2.,4.) circle (2.0pt);
		\draw[color=ududff] (2.1850935904707884,4.3801361315938765) node {$x_3$};
		\draw [fill=ududff] (2.,2.) circle (2.0pt);
		\draw[color=ududff] (2.1850935904707884,2.3835394214407284) node {$x_4$};
		\draw [fill=ududff] (4.,2.) circle (2.0pt);
		\draw[color=ududff] (3.802336925694838,2.234032898468545) node {$x_5$};
		\draw [fill=ududff] (5.,3.) circle (2.0pt);
		\draw[color=ududff] (4.461213840045377,3.2620419739081137) node {$x_1$};
		\draw [fill=ududff] (4.,4.) circle (2.0pt);
		\draw[color=ududff] (4.181690300623936,4.3801361315938765) node {$x_2$};
		\draw [fill=ududff] (6.,3.) circle (2.0pt);
		\draw[color=ududff] (5.938695405558707,2.742926829268295) node {$z_1$};
		\draw [fill=ududff] (7.,3.) circle (2.0pt);
		\draw[color=ududff] (7.556279069767441,3.2620419739081137) node {$y_1$};
		\draw [fill=ududff] (8.,4.) circle (2.0pt);
		\draw[color=ududff] (8.394509359047078,4.46) node {$y_2$};
		\draw [fill=ududff] (10.,4.) circle (2.0pt);
		\draw[color=ududff] (10.19144639818491,4.3801361315938765) node {$y_3$};
		\draw [fill=ududff] (10.,2.) circle (2.0pt);
		\draw[color=ududff] (10.19144639818491,2.3835394214407284) node {$y_4$};
		\draw [fill=ududff] (8.,2.) circle (2.0pt);
		\draw[color=ududff] (8.194849688031763,2.3835394214407284) node {$y_5$};
		\draw [fill=ududff] (8.,5.) circle (2.0pt);
		\draw[color=ududff] (8.194849688031763,5.37843448667045) node {$z_2$};
		\end{scriptsize}
		\end{tikzpicture}
	\end{center}
	Then we have $\reg{I(G)}=5$ and $\nu(G)=4$.	
\end{exmp}

\begin{exmp}
	Statement $\textit{(IV)}$ of \autoref{full_characterization}.
	Let $G$ be the graph below.
	\begin{center}
		\definecolor{ududff}{rgb}{0.30196078431372547,0.30196078431372547,1.}
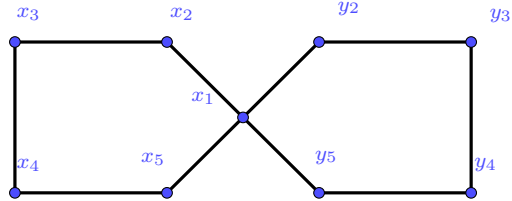
		\begin{tikzpicture}[line cap=round,line join=round,>=triangle 45,x=1.0cm,y=1.0cm]
		\clip(1.5,1.5) rectangle (8.8,4.8);
		\draw [line width=1.2pt] (2.,4.)-- (2.,2.);
		\draw [line width=1.2pt] (2.,2.)-- (4.,2.);
		\draw [line width=1.2pt] (4.,2.)-- (5.,3.);
		\draw [line width=1.2pt] (5.,3.)-- (4.,4.);
		\draw [line width=1.2pt] (4.,4.)-- (2.,4.);
		\draw [line width=1.2pt] (5.,3.)-- (6.,4.);
		\draw [line width=1.2pt] (6.,4.)-- (8.,4.);
		\draw [line width=1.2pt] (8.,4.)-- (8.,2.);
		\draw [line width=1.2pt] (8.,2.)-- (6.,2.);
		\draw [line width=1.2pt] (6.,2.)-- (5.,3.);
		\begin{scriptsize}
		\draw [fill=ududff] (2.,4.) circle (2.0pt);
		\draw[color=ududff] (2.1814634146341505,4.371060692002271) node {$x_3$};
		\draw [fill=ududff] (2.,2.) circle (2.0pt);
		\draw[color=ududff] (2.1814634146341505,2.3744639818491238) node {$x_4$};
		\draw [fill=ududff] (4.,2.) circle (2.0pt);
		\draw[color=ududff] (3.818672716959732,2.4543278502552495) node {$x_5$};
		\draw [fill=ududff] (5.,3.) circle (2.0pt);
		\draw[color=ududff] (4.477549631310271,3.2529665343165086) node {$x_1$};
		\draw [fill=ududff] (4.,4.) circle (2.0pt);
		\draw[color=ududff] (4.19802609188883,4.371060692002271) node {$x_2$};
		\draw [fill=ududff] (6.,4.) circle (2.0pt);
		\draw[color=ududff] (6.394282473057293,4.450924560408398) node {$y_2$};
		\draw [fill=ududff] (8.,4.) circle (2.0pt);
		\draw[color=ududff] (8.39087918321044,4.35109472490074) node {$y_3$};
		\draw [fill=ududff] (8.,2.) circle (2.0pt);
		\draw[color=ududff] (8.191219512195126,2.3744639818491238) node {$y_4$};
		\draw [fill=ududff] (6.,2.) circle (2.0pt);
		\draw[color=ududff] (6.094792966534321,2.4742938173567812) node {$y_5$};
		\end{scriptsize}
		\end{tikzpicture}
	\end{center}
	Then we have $\reg{I(G)}=4$ and $\nu(G)=2$.
	
	By adding an edge, let $G$ be the graph below. 
	\begin{center}
		\definecolor{ududff}{rgb}{0.30196078431372547,0.30196078431372547,1.}
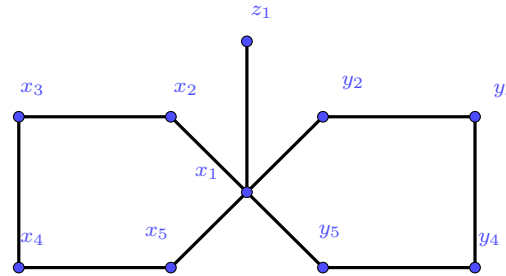
		\begin{tikzpicture}[line cap=round,line join=round,>=triangle 45,x=1.0cm,y=1.0cm]
		\clip(1.5,1.5) rectangle (8.7,5.7);
		\draw [line width=1.2pt] (2.,4.)-- (2.,2.);
		\draw [line width=1.2pt] (2.,2.)-- (4.,2.);
		\draw [line width=1.2pt] (4.,2.)-- (5.,3.);
		\draw [line width=1.2pt] (5.,3.)-- (4.,4.);
		\draw [line width=1.2pt] (4.,4.)-- (2.,4.);
		\draw [line width=1.2pt] (5.,3.)-- (6.,4.);
		\draw [line width=1.2pt] (6.,4.)-- (8.,4.);
		\draw [line width=1.2pt] (8.,4.)-- (8.,2.);
		\draw [line width=1.2pt] (8.,2.)-- (6.,2.);
		\draw [line width=1.2pt] (6.,2.)-- (5.,3.);
		\draw [line width=1.2pt] (5.,3.)-- (5.,5.);
		\begin{scriptsize}
		\draw [fill=ududff] (2.,4.) circle (2.0pt);
		\draw[color=ududff] (2.1814634146341505,4.371060692002271) node {$x_3$};
		\draw [fill=ududff] (2.,2.) circle (2.0pt);
		\draw[color=ududff] (2.1814634146341505,2.3744639818491238) node {$x_4$};
		\draw [fill=ududff] (4.,2.) circle (2.0pt);
		\draw[color=ududff] (3.818672716959732,2.4543278502552495) node {$x_5$};
		\draw [fill=ududff] (5.,3.) circle (2.0pt);
		\draw[color=ududff] (4.477549631310271,3.2529665343165086) node {$x_1$};
		\draw [fill=ududff] (4.,4.) circle (2.0pt);
		\draw[color=ududff] (4.19802609188883,4.371060692002271) node {$x_2$};
		\draw [fill=ududff] (6.,4.) circle (2.0pt);
		\draw[color=ududff] (6.394282473057293,4.450924560408398) node {$y_2$};
		\draw [fill=ududff] (8.,4.) circle (2.0pt);
		\draw[color=ududff] (8.39087918321044,4.35109472490074) node {$y_3$};
		\draw [fill=ududff] (8.,2.) circle (2.0pt);
		\draw[color=ududff] (8.191219512195126,2.3744639818491238) node {$y_4$};
		\draw [fill=ududff] (6.,2.) circle (2.0pt);
		\draw[color=ududff] (6.094792966534321,2.4742938173567812) node {$y_5$};
		\draw [fill=ududff] (5.,5.) circle (2.0pt);
		\draw[color=ududff] (5.196324446965404,5.389325014180377) node {$z_1$};
		\end{scriptsize}
		\end{tikzpicture}
	\end{center}
	Then we have $\reg{I(G)}=4$ and $\nu(G)=3$.
\end{exmp}

\section{Castelnuovo-Mumford regularity of powers}\label{Upper_bound_section}

In this section, we study  the regularity of the powers of $ \edge{\cnplcm} $ when $ l\leq2 $. 
Our strategy is to obtain a lower bound and an upper bound for $\reg{\edge{\cnplcm}^q}$, such that both coincide and are equal to  $2q+\reg {\edge{\cnplcm}}$.
To obtain the upper bound, we follow the argument of  Banerjee from \cite[Theorem 5.2]{banerjee}. 
To calculate the lower bound, we proceed by looking at ``nice" induced subgraphs of $ \cnplcm $.

As a side result, we answer an interesting question on the behavior of the constant term of the asymptotically linear regularity function.
Let $I$ be an arbitrary ideal generated in degree $d$ and let $b_q:=\reg (I^q)- dq$ for $ q\geq 1 $. 
An  interesting question is to study of the sequence $ \{b_i\}_{i\geq1} $. 
In \cite{Eisenbud-Harris} Eisenbud and Harris proved that if $\dim (R/I)=0$, then $ \{b_i\}_{i\geq1} $ is a weakly decreasing sequence of non-negative integers. 
In \cite{BBH} Banerjee, Beyarslan and  H\`{a} conjectured that for any edge ideal, $ \{b_i\}_{i\geq1} $ is a weakly decreasing sequence  (see \cite[Conjecture 7.11]{BBH}).  
For the edge ideal of any dumbbell graph with $ l\leq2 $, we prove $b_i=b_1$ for all $i\geq 1$. However, we expect $b_i\leq b_1$ for all $i\geq 1$ for any graph. 

\begin{rem}\label{remark n+m+1}
	From \autoref{formula_nu_dumbbell} and  \autoref{reg_dumbbell}, for any $l \le 2$ we have that 
	$$
	\reg{I(\cnplcm)} \ge \lfloor \frac{n+m+l+1}{3} \rfloor.
	$$
\end{rem}
The previous inequality is not satisfied when $l\ge 3$, because $\reg{I(C_4 \cdot P_3 \cdot C_4)} = 3$ and $\lfloor \frac{4+4+3+1}{3} \rfloor =4$.

As recalled earlier, we use the notation of even-connection from Banerjee \cite[Theorem 5.2]{banerjee}. 
The following lemma is crucial in our treatment of the even-connected vertices, and its proof is similar to \cite[Lemma 6.13]{banerjee}.

\begin{lem}\label{claim cncm}
	Let $G$ be a graph. 
	As in \autoref{newGraphBanerjee}, let $G^\prime$ be the graph associated to ${(I(G)^{q+1} \colon e_1\cdots e_q)}^{\text{pol}}$.
	Suppose $u=p_0,p_1,\ldots,p_{2s+1}=v$ is a path that even-connects $u$ and $v$ with respect to the $q$-fold $e_1\cdots e_q$.
	Then we have
	$$ \bigcup_{i=0}^{2s+1}  N_{G'}[p_i] \,\subset\, N_{G'}[u]\cup N_{G'}[v]. $$
	\begin{proof}
		Let $U$ be the set of vertices $U=\{p_0,p_1,\ldots,p_{2s+1}\}$.
		For each $1 \le k \le s$ we have that $p_{2k-1}p_{2k}=e_{j_k}$ for some $1 \le j_k \le q$, i.e. $u$ and $v$ are even connected with respect to the $s$-fold $e_{j_1}e_{j_2}\cdots e_{j_s}$.
		
		Let $w$ be a vertex even-connected to some vertex $z \in U$ with respect to the $q$-fold $e_1\cdots e_q$.
		Then, there exists a path $z=r_0,r_1,\ldots,r_{2t+1}=w$ that even-connects $z$ and $w$ with respect to the $q$-fold $e_1\cdots e_q$.
		Let $i$ be the largest integer such that $r_i \in U$.
		From the fact that $r_0=z \in U$, we have that the integer $i$ is well defined and $i\ge 0$.
		Let $k$ be an integer such that $p_k=r_i$.
		
		The proof is now divided into four different cases depending on $i \,\text{mod}\, 2$ and $k \,\text{mod}\, 2$. 
		When $i$ and $k$ are both odd integers, we have that $r_{i}r_{i+1}$ is equal to some edge of $\{e_1,e_2,\ldots,e_q\}$ and that $p_{k-1}p_{k}$ is not equal to any edge of $\{ e_{j_1},e_{j_2},\ldots e_{j_s} \}$.
		By the definition of $i$ we have
		$$
		\{r_{i+1},r_{i+2},\ldots,r_{2t+1}\} \cap U = \emptyset.
		$$
		So, in this case, it follows that 
		$$
		u = p_0, \ldots, p_{k-1}, p_k = r_i, r_{i+1}, \ldots, r_{2t+1} = w
		$$
		is a path that even-connects $u$ and $w$ with respect to the $q$-fold $e_1\cdots e_q$.
		
		The other three cases follow in a similar way.
	\end{proof}
\end{lem}

 \begin{rem}\label{remark leaf}
 Let $G= \cnplcm$. If $(\edge {G}^{q+1}\colon e_1\cdots e_q)$ is not a square-free monomial ideal and  $G'$ is the associated graph, then there exist a vertex $x_i$ which is even-connected to itself. 
 Therefore $G'$ has a leaf. 
 By \autoref{claim cncm} one can see $N_{G'}[x_i]$ contains one of the two cycles. 
 In particular, if we denote the leaf by $e$, then $G'_{e}$ is an induced subgraph of a unicyclic graph.
 \end{rem}

\begin{theorem}\label{reg power cnp1cm}
Let $G= \cnplcm$ and  $I=I(G)$ be its edge ideal, then
	$$
	\reg {(I^{q+1} \colon e_1\cdots e_q)}\leq \reg I
	$$
	for any $1\leq q$ and  any edges $e_1,\dots , e_q \in E(G)$.
\end{theorem}
\begin{proof}
We split the proof into two cases. 

\underline{Case 1.} First, suppose $(I^{q+1} \colon e_1\cdots e_q)$ is a square-free monomial ideal. In this case $(I^{q+1} \colon e_1\cdots e_q)= I(G^\prime)$ where $ G^\prime $ is a graph with $ V(G)=V(G^\prime) $ and $E(G)\subseteq E(G^\prime)$. Let $E(G^\prime)= E(G)\cup \lbrace a_1,\dots, a_r\rbrace$. By \autoref{short exact sequence}, we have
$$
\reg I(G')\leq \max \lbrace \reg I(G'\setminus a_1), \reg I(G'_{a_1})+1\rbrace
$$
From \autoref{claim cncm}, $G'_{a_1}$ is obtained from $G'$ by removing 
one of the cycles or deleting at least 6 vertices. 

Suppose $G'_{a_1}$ is obtained by removing one of the cycles.
Without loss of generality assume that $C_n$ is deleted, then there exists a Hamiltonian path of length $\leq m$ when $l=2$ and of length $\leq m-1$ when $l=1$. 
From \autoref{Hamiltonian} and \autoref{remark n+m+1}, if $C_n$ has  $n\ge 4$ vertices, then we have $\reg I(G'_{a_1})\leq \reg I(G)-1$.
In the case $n=3$, there is a Hamiltonian path of length $\leq m-3$, and so \autoref{Hamiltonian} and \autoref{remark n+m+1} again imply $\reg I(G'_{a_1})\leq \reg I(G)-1$.

Suppose $G'_{a_1}$ is obtained by removing at least 6 vertices.
Let $H^\prime$ be the graph given by deleting $N_G[a_1]$.
From the assumption of deleting at least 6 vertices we have that $\lvert H^\prime \rvert \le \lvert G \rvert-6\le n+m+l-8$.
We note that we can add two vertices to $H^\prime$ and connect them in such a way that we obtain a Hamiltonian path.
Let $H$ be a graph obtained by adding two vertices and certain edges connecting these two new vertices, such that $H$ has a Hamiltonian path.
Note that $G^\prime_{a_1}$ is an induced subgraph of $H$.
Since $\lvert H \rvert \le n+m+l-6$,  \autoref{Hamiltonian} yields
 $$
 \reg I(H) \leq \lfloor \dfrac{n+m+l-5}{3}\rfloor +1= \lfloor \dfrac{n+m+l+1}{3}\rfloor-1.
 $$ 
 Applying \autoref{remark n+m+1}, we get 
 $$
 \reg  I(G'_{a_1})\leq \reg I(H) \leq \reg I(G)-1.
 $$ 
 Therefore 
 $$\reg I(G')\leq \max \lbrace \reg I(G'\setminus a_1), \reg I(G)\rbrace.
 $$
 
 In the same way, for any subgraph $H = G^\prime \setminus \{a_1,\ldots,a_i\}$, we have that 
 $$
 \reg( I(H_{a_{i+1}}) ) \le \reg(I(G)) - 1.
 $$
 So, we also obtain
 $$
 \reg I(G'\setminus a_1)\leq \max \lbrace \reg I(G'\setminus \lbrace a_1,a_2\rbrace), \reg I(G)\rbrace.
 $$ 
 By continuing this process, we get $\reg I(G')\leq \reg I(G)$.

 \underline{Case 2.}
Suppose $(I^{q+1} \colon e_1\cdots e_q)$ is not square-free and $G'$ is the graph associated to ${(I^{q+1} \colon e_1\cdots e_q)}^\text{pol}$.  
Let $\{b_1,b_2,\ldots,b_s\}$ be the subset of edges of $E(G^\prime)\setminus E(G)$ that are generated by square monomials , i.e. each $b_i$ is a whisker. 

From \autoref{short exact sequence} we have the inequality 
$$
\reg{I(G^\prime)} \le \max\{ \reg{I(G^\prime \setminus b_1)}, 1 + \reg{I(G^\prime_{b_1})} \}.
$$
\autoref{remark leaf} implies that one of the cycles is deleted from $G^\prime_{b_1}$, then there exists  an edge $e\in G$ such that $d(e, G^\prime_{b_1})\ge 2$.
So, for such an edge $e$ we get that the disjoint union $G^\prime_{b_1} \cup e$ is an induced subgraph of $G^\prime \setminus b_1$.
Thus, \autoref{short exact sequence} and \cite[Lemma 3.2]{HOA_TAM} yield that
$$
\reg(I(G^\prime_{b_1})) + 1 = \reg(I(G^\prime_{b_1} \cup e)) \,\le\, \reg(I(G^\prime)).
$$
Therefore, we obtain that $\reg{I(G^\prime)} \le \reg{I(G^\prime \setminus b_1)}$.

By applying the same argument, it follows that 
$$
\reg{I(G^\prime)} \le \reg{I(G^\prime \setminus b_1)} \le \reg{I(G^\prime \setminus \{b_1,b_2\})} \le \cdots \le \reg{I(G^\prime \setminus \{b_1,\ldots,b_s\})}.
$$
Since the graph $G^\prime \setminus \{b_1,\ldots,b_s\}$ has no whiskers, then \underline{Step 1} implies that 
 $$
 \reg{I(G^\prime)} \le \reg{I(G^\prime \setminus \{b_1,\ldots,b_s\})} \le \reg{I(G)}.
 $$
Therefore, the proof is completed.
\end{proof}

\begin{rem}
The previous theorem is a generalization of a work done by Yan Gu in \cite{YanGu} for the case $l=1$.
\end{rem}
\begin{theorem}
	\label{powers_lower_bound}
	For the dumbbell graph $\cnplcm$ with $l\le 2$, we have
	$$
	\reg{{I(\cnplcm)}^q} \geq 2q + \reg{I(\cnplcm)} - 2,
	$$
	for any $q \ge 1$.
	\end{theorem}
	\begin{proof}
				Using the inequality $\reg{{I(\cncm)}^q} \ge 2q + \nu(\cncm) - 1$ of \cite[Theorem 4.5]{TAI_FOREST_CYCLE}, for the cases where $\reg{I(\cnplcm)} = \nu(\cnplcm) + 1$ we get the expected inequality.
		We divide the proof in two halves, the cases $l=1$ and $l=2$. 
		
		\underline{Case 1.} Let $l=1$. We only need to focus on the case where $n, m \equiv 2 \; \;(\Mod 3)$.
		Let $H$ be the induced subgraph of $\cnlcm{1}$ mentioned in the proof of \autoref{reg of cn p1 cm}, i.e. $H=(C_n \cdot P_1 \cdot C_m) \setminus \{x_n\}=P_{n-1}\cdot C_m$.
		Using \autoref{formula_nu_dumbbell}, \autoref{nu_path_cycle} and the modularity $n, m \equiv 2 \; (\Mod 3)$,
		we can check that 
		$$
		\nu(H)=\nu(\cnlcm{1})
		$$
		and that 
		$$
		\nu(H) = \nu(H \setminus \Gamma_{H}(C_m)).
		$$
		From \autoref{reg of cn p1 cm} and \cite[Theorem 1.1]{REG_UNICYCLIC_GRAPH} we get 
		$$
		\reg{I(\cnlcm{1})}=\nu(\cnlcm{1})+2=\nu(H)+2=\reg{I(H)}.
		$$
		Since $H$ is an induced subgraph of $\cnlcm{1}$, then from \cite[Theorem 1.2]{REG_UNICYCLIC_GRAPH} and \cite[Corollay 4.3]{TAI_FOREST_CYCLE} we get the inequality
		$$
		\reg{I(\cnlcm{1})^q} \ge \reg{I(H)^q}=2q+\reg{I(H)}-2=2q+\reg{I(\cnlcm{1})}-2.
		$$
		
		\underline{Case 2.} Let $l=2$. We only need to focus on the cases where $n \equiv 0,1 \; (\Mod 3)$ and $m \equiv 2 \;(\Mod 3)$. 
		We take the same induced subgraph $H$ as in \autoref{reg_case_zero_two}.
		The induced subgraph $H = (\cncm) \setminus \{x_1\}$ of $\cncm$ is given as the union of a path of length $n-1$ and the cycle $C_m$, i.e., $H = P_{n-1} \cup C_m$. 
		
		By \autoref{formula_reg_dumbbell}, for the cases $n \equiv 0,1 \; (\Mod 3)$ and $m \equiv 2 \;(\Mod 3)$, we have $$\reg{I(\cncm)}=\nu(\cncm)+2=\lfloor \frac{n}{3} \rfloor + \lfloor \frac{m}{3} \rfloor + 2,$$ and from \cite[Theorem 1.1]{REG_UNICYCLIC_GRAPH}
		we have $$\reg{I(H)}=\nu(H)+2=\nu(P_{n-1})+\nu(C_m)+2=\lfloor \frac{n}{3} \rfloor + \lfloor \frac{m}{3} \rfloor + 2.$$
		Hence, we get $\reg{I(\cncm)}=\reg{I(H)}$.
		Finally, using \cite[Theorem 1.2]{REG_UNICYCLIC_GRAPH} and \cite[Corollary 4.3]{TAI_FOREST_CYCLE}, we get the inequality
		$$
		\reg{I(\cncm)^q} \ge \reg{I(H)^q}=2q+\reg{I(H)}-2=2q+\reg{I(\cncm)}-2.
		$$
		Therefore, the proof is completed.
	\end{proof}

\begin{theorem}
	\label{Equality}
	For the dumbbell graph $\cnplcm$ with $l\leq 2$,  we have
	$$
	\reg {I(\cnplcm)^q}=  2q + \reg{I(\cnplcm)} -2
	$$
	for all $q\geq 1$.
\end{theorem}
\begin{proof}
	It follows by  \autoref{reg power cnp1cm}, \autoref{banerjee} and \autoref{powers_lower_bound}.
\end{proof}

\begin{rem}\label{Equality_counter_example}
	One may ask whether 
	$$\reg{I(\cnlcm{ l})^q}=2q + \reg{I(\cnlcm{l})} -2$$
	always holds for given $n,m,l$ and $q$. Unfortunately, this is not the case. In fact, it can be checked that 
	$$
	6=\reg{I(C_5\cdot P_3 \cdot C_5)^2}<4+ \reg{I(C_5 \cdot P_3 \cdot C_5)} -2=7.
	$$
\end{rem}

\section*{Acknowledgments}
This project is originated from the summer school ``Pragmatic 2017". The authors would like to sincerely express their gratitude to the organizers Alfio Ragusa, Elena Guardo, Francesco Russo, and Giuseppe Zappal\`{a}, and to the lecturers Brian Harbourne, Adam Van Tuyl, Enrico Carlini, and T\`{a}i Huy H\`{a}. 
We are deeply grateful to the last lecturer for introducing this topic to us and for his mentoring. 
We are grateful to the referee for valuable comments and suggestions
that improved this paper in many ways.
We thank Yan Gu for pointing out an error in the initial version. 
The computer algebra system Macaulay2 \cite{M2}, was very helpful to compute several examples in the preparation of this paper.

\bibliographystyle{elsarticle-num} 
\bibliography{references}

\end{document}